\documentclass[reqno]{amsart}
\usepackage{amssymb,latexsym,amsmath,amsthm,enumerate,amsbsy}
\usepackage[mathscr]{eucal}
\usepackage{framed,color,graphicx}
\usepackage{mathrsfs}
\usepackage[all]{xy}
\usepackage{tikz}
\usetikzlibrary{positioning,decorations.pathreplacing,patterns,decorations.pathmorphing}
\tikzset{%
element/.style={draw, shape=circle, fill=white, inner sep=1.4pt}
}

\DeclareSymbolFont{bbold}{U}{bbold}{m}{n}
\DeclareSymbolFontAlphabet{\mathbbold}{bbold}

\theoremstyle{plain}
\newtheorem{thm}{Theorem}[section]
\newtheorem{lem}[thm]{Lemma}
\newtheorem{cor}[thm]{Corollary}
\newtheorem{pro}[thm]{Proposition}

\theoremstyle{definition}

\newtheorem{remark}[thm]{Remark}

\newcommand{\bp}{\mathbf{p}}
\newcommand{\bq}{\mathbf{q}}

\newcommand{\bu}{\mathbf{u}}

\newcommand{\bw}{\mathbf{w}}

\begin{document}

\title[The finite basis problem for ai-semirings of order four]
{The finite basis problem for additively idempotent semirings of order four, II}

\author{Mengya Yue}
\address{School of Mathematics, Northwest University, Xi'an, 710127, Shaanxi, P.R. China}
\email{myayue@yeah.net}

\author{Miaomiao Ren}
\address{School of Mathematics, Northwest University, Xi'an, 710127, Shaanxi, P.R. China}
\email{miaomiaoren@yeah.net}

\author{Lingli Zeng}
\address{School of Mathematics, Northwest University, Xi'an, 710127, Shaanxi, P.R. China}
\email{zengll929@nwu.edu.cn}

\author{Yong Shao}
\address{School of Mathematics, Northwest University, Xi'an, 710127, Shaanxi, P.R. China}
\email{yongshaomath@126.com}

\thanks{Miaomiao Ren, corresponding author, is supported by National Natural Science Foundation of China (12371024).
}

\begin{abstract}
We study the finite basis problem for $4$-element additively idempotent semirings whose
additive reducts are quasi-antichains. Up to isomorphism, there are $93$ such algebras.
We show that with the exception of the semiring $S_{(4, 435)}$, all of them are finitely based.
\end{abstract}

\maketitle

\section{Introduction and preliminaries}
The present paper is a continuation of \cite{rlzc}.
We study the finite basis problem for $4$-element additively idempotent semirings whose
additive reducts are quasi-antichains.
The readers are referred to \cite{rlzc} for necessary background, motivation, references, techniques and notations.

Up to isomorphism, there are 6 ai-semirings of order two,
which are denoted by $L_2$, $R_2$, $M_2$, $D_2$, $N_2$ and $T_2$.
The solution of the equational problem for these algebras can be found in \cite[Lemma 1.1]{sr}
and will be repeatedly used in the sequel.
Up to isomorphism, there are 61 ai-semirings of order three,
which are denoted by $S_i$, $1 \leq i \leq 61$.
One can find detailed information on these algebras in \cite{zrc}.

Up to isomorphism,
there are $866$ ai-semirings of order four.
The additive reducts of $58$ of them are semilattices of height $1$.
Ren et al. \cite{rlzc} has answered the finite basis problem for these algebras.
The additive reducts of $93$ of them are quasi-antichains. These algebras are denoted by $S_{(4, k)}$, $388\leq k \leq 480$.
We assume that the carrier set of these semirings is $\{1, 2, 3, 4\}$.
Their Cayley tables for multiplication are listed in Table \ref{tb1},
and their Cayley tables for addition are determined by Figure \ref{figure01}.

\setlength{\unitlength}{1cm}
\begin{figure}[ht]\label{figure01}
\begin{picture}(35, 2.25)
\put(6.5,2.2){\line(1,-1){1}}
\put(6.5,2.2){\line(-1,-1){1}}
\put(6.5,0.2){\line(1,1){1}}
\put(6.5,0.2){\line(-1,1){1}}

\multiput(6.5,2.2)(1,-1){2}{\circle*{0.1}}
\multiput(6.5,2.2)(-1,-1){2}{\circle*{0.1}}
\multiput(6.5,0.2)(1,1){2}{\circle*{0.1}}

\put(6.5,2.4){\makebox(0,0){$1$}}
\put(6.5,0){\makebox(0,0){$2$}}
\put(5.3,1.2){\makebox(0,0){$3$}}
\put(7.7,1.2){\makebox(0,0){$4$}}
\end{picture}
\caption{The additive order of $S_{(4, k)}$, $388\leq k \leq 480$}
\end{figure}
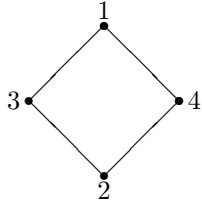
The following theorem is the main result of this paper.
Its proof will be completed in the following sections.

\begin{table}[ht]
\caption{The multiplicative tables of $4$-element ai-semirings whose additive reducts are quasi-antichains} \label{tb1}
\begin{tabular}{cccccc}
\hline
Semiring & $\cdot$ & Semiring & $\cdot$ & Semiring & $\cdot$\\
\hline
$S_{(4, 388)}$
&
\begin{tabular}{cccc}
1 & 1 & 1 & 1\\
1 & 1 & 1 & 1 \\
1 & 1 & 1 & 1 \\
1 & 1 & 1 & 1 \\
\end{tabular}
&
$S_{(4, 389)}$
&
\begin{tabular}{cccc}
1 & 1 & 1 & 1\\
1 & 2 & 1 & 1 \\
1 & 1 & 1 & 1 \\
1 & 1 & 1 & 1 \\
\end{tabular}
&
$S_{(4, 390)}$
&
\begin{tabular}{cccc}
1 & 1 & 1 & 1\\
1 & 2 & 1 & 4 \\
1 & 1 & 1 & 1 \\
1 & 1 & 1 & 1 \\
\end{tabular}\\
\hline

$S_{(4, 391)}$
&
\begin{tabular}{cccc}
1 & 1 & 1 & 1\\
1 & 2 & 3 & 4 \\
1 & 1 & 1 & 1 \\
1 & 1 & 1 & 1 \\
\end{tabular}
&
$S_{(4, 392)}$
&
\begin{tabular}{cccc}
1 & 1 & 1 & 1\\
1 & 2 & 1 & 2 \\
1 & 1 & 1 & 1 \\
1 & 2 & 1 & 2 \\
\end{tabular}
&
$S_{(4, 393)}$
&
\begin{tabular}{cccc}
1 & 1 & 1 & 1\\
1 & 2 & 1 & 2 \\
1 & 1 & 1 & 1 \\
1 & 2 & 1 & 4 \\
\end{tabular}\\
\hline

$S_{(4, 394)}$
&
\begin{tabular}{cccc}
1 & 1 & 1 & 1\\
1 & 2 & 1 & 4 \\
1 & 1 & 1 & 1 \\
1 & 2 & 1 & 4 \\
\end{tabular}
&
$S_{(4, 395)}$
&
\begin{tabular}{cccc}
1 & 1 & 1 & 1\\
1 & 2 & 3 & 4 \\
1 & 1 & 1 & 1 \\
1 & 2 & 3 & 4 \\
\end{tabular}
&
$S_{(4, 396)}$
&
\begin{tabular}{cccc}
1 & 1 & 1 & 1\\
1 & 2 & 1 & 1 \\
1 & 1 & 1 & 1 \\
1 & 4 & 1 & 1 \\
\end{tabular}\\
\hline

$S_{(4, 397)}$
&
\begin{tabular}{cccc}
1 & 1 & 1 & 1\\
1 & 2 & 1 & 2 \\
1 & 1 & 1 & 1 \\
1 & 4 & 1 & 4 \\
\end{tabular}
&
$S_{(4, 398)}$
&
\begin{tabular}{cccc}
1 & 1 & 1 & 1\\
1 & 2 & 1 & 4 \\
1 & 1 & 1 & 1 \\
1 & 4 & 1 & 1 \\
\end{tabular}
&
$S_{(4, 399)}$
&
\begin{tabular}{cccc}
1 & 1 & 1 & 1\\
1 & 2 & 1 & 4 \\
1 & 1 & 1 & 1 \\
1 & 4 & 1 & 4 \\
\end{tabular}\\
\hline
$S_{(4, 400)}$
&
\begin{tabular}{cccc}
1 & 1 & 1 & 1\\
1 & 2 & 3 & 1 \\
1 & 1 & 1 & 1 \\
1 & 4 & 1 & 1 \\
\end{tabular}
&
$S_{(4, 401)}$
&
\begin{tabular}{cccc}
1 & 1 & 1 & 1\\
1 & 2 & 3 & 4 \\
1 & 1 & 1 & 1 \\
1 & 4 & 1 & 1 \\
\end{tabular}
&
$S_{(4, 402)}$
&
\begin{tabular}{cccc}
1 & 1 & 1 & 1\\
1 & 2 & 3 & 4 \\
1 & 1 & 1 & 1 \\
1 & 4 & 1 & 4 \\
\end{tabular}\\
\hline

$S_{(4, 403)}$
&
\begin{tabular}{cccc}
1 & 1 & 1 & 1\\
1 & 2 & 1 & 1 \\
1 & 3 & 1 & 1 \\
1 & 4 & 1 & 1 \\
\end{tabular}
&
$S_{(4, 404)}$
&
\begin{tabular}{cccc}
1 & 1 & 1 & 1\\
1 & 2 & 1 & 2 \\
1 & 3 & 1 & 3 \\
1 & 4 & 1 & 4 \\
\end{tabular}
&
$S_{(4, 405)}$
&
\begin{tabular}{cccc}
1 & 1 & 1 & 1\\
1 & 2 & 1 & 4 \\
1 & 3 & 1 & 1 \\
1 & 4 & 1 & 1 \\
\end{tabular}\\
\hline

$S_{(4, 406)}$
&
\begin{tabular}{cccc}
1 & 1 & 1 & 1\\
1 & 2 & 1 & 4 \\
1 & 3 & 1 & 1 \\
1 & 4 & 1 & 4 \\
\end{tabular}
&
$S_{(4, 407)}$
&
\begin{tabular}{cccc}
1 & 1 & 1 & 1\\
1 & 2 & 3 & 4 \\
1 & 3 & 1 & 1 \\
1 & 4 & 1 & 1 \\
\end{tabular}
&
$S_{(4, 408)}$
&
\begin{tabular}{cccc}
1 & 1 & 1 & 1\\
1 & 2 & 3 & 4 \\
1 & 3 & 1 & 1 \\
1 & 4 & 1 & 4 \\
\end{tabular}\\
\hline

$S_{(4, 409)}$
&
\begin{tabular}{cccc}
1 & 1 & 1 & 1\\
1 & 2 & 3 & 4 \\
1 & 3 & 3 & 1 \\
1 & 4 & 1 & 4 \\
\end{tabular}
&
$S_{(4, 410)}$
&
\begin{tabular}{cccc}
1 & 1 & 1 & 1\\
1 & 3 & 1 & 1 \\
1 & 1 & 1 & 1 \\
1 & 1 & 1 & 1 \\
\end{tabular}
&
$S_{(4, 411)}$
&
\begin{tabular}{cccc}
1 & 1 & 1 & 1\\
1 & 3 & 1 & 3 \\
1 & 1 & 1 & 1 \\
1 & 1 & 1 & 1 \\
\end{tabular}\\
\hline

$S_{(4, 412)}$
&
\begin{tabular}{cccc}
1 & 1 & 1 & 1\\
1 & 3 & 1 & 1 \\
1 & 1 & 1 & 1 \\
1 & 3 & 1 & 1 \\
\end{tabular}
&
$S_{(4, 413)}$
&
\begin{tabular}{cccc}
1 & 1 & 1 & 1\\
1 & 3 & 1 & 3 \\
1 & 1 & 1 & 1 \\
1 & 3 & 1 & 1 \\
\end{tabular}
&
$S_{(4, 414)}$
&
\begin{tabular}{cccc}
1 & 1 & 1 & 1\\
1 & 3 & 1 & 3 \\
1 & 1 & 1 & 1 \\
1 & 3 & 1 & 3 \\
\end{tabular}\\
\hline

$S_{(4, 415)}$
&
\begin{tabular}{cccc}
1 & 1 & 1 & 1\\
1 & 3 & 3 & 1 \\
1 & 3 & 3 & 1 \\
1 & 1 & 1 & 1 \\
\end{tabular}
&
$S_{(4, 416)}$
&
\begin{tabular}{cccc}
1 & 2 & 1 & 1\\
1 & 2 & 1 & 1 \\
1 & 2 & 1 & 1 \\
1 & 2 & 1 & 1 \\
\end{tabular}
&
$S_{(4, 417)}$
&
\begin{tabular}{cccc}
1 & 2 & 1 & 1\\
1 & 2 & 1 & 2 \\
1 & 2 & 1 & 1 \\
1 & 2 & 1 & 4 \\
\end{tabular}\\
\hline

$S_{(4, 418)}$
&
\begin{tabular}{cccc}
1 & 2 & 1 & 1\\
1 & 2 & 1 & 2 \\
1 & 2 & 1 & 3 \\
1 & 2 & 1 & 4 \\
\end{tabular}
&
$S_{(4, 419)}$
&
\begin{tabular}{cccc}
1 & 2 & 1 & 2\\
1 & 2 & 1 & 2 \\
1 & 2 & 1 & 2 \\
1 & 2 & 1 & 2 \\
\end{tabular}
&
$S_{(4, 420)}$
&
\begin{tabular}{cccc}
1 & 2 & 1 & 4 \\
1 & 2 & 1 & 4 \\
1 & 2 & 1 & 4 \\
1 & 2 & 1 & 4 \\
\end{tabular}\\
\hline
\end{tabular}
\end{table}

\newpage

\begin{table}[ht]
\begin{tabular}{cccccc}
\hline
$S_{(4, 421)}$
&
\begin{tabular}{cccc}
1 & 2 & 3 & 4\\
1 & 2 & 3 & 4 \\
1 & 2 & 3 & 4 \\
1 & 2 & 3 & 4 \\
\end{tabular}
&
$S_{(4, 422)}$
&
\begin{tabular}{cccc}
1 & 3 & 3 & 1 \\
1 & 2 & 3 & 4 \\
1 & 3 & 3 & 1 \\
1 & 2 & 3 & 4 \\
\end{tabular}
&
$S_{(4, 423)}$
&
\begin{tabular}{cccc}
1 & 3 & 3 & 1 \\
1 & 2 & 3 & 1 \\
1 & 3 & 3 & 1 \\
1 & 3 & 3 & 1 \\
\end{tabular}\\
\hline
$S_{(4, 424)}$
&
\begin{tabular}{cccc}
1 & 3 & 3 & 1\\
1 & 2 & 3 & 4 \\
1 & 3 & 3 & 1 \\
1 & 3 & 3 & 1 \\
\end{tabular}
&
$S_{(4, 425)}$
&
\begin{tabular}{cccc}
1 & 3 & 3 & 1 \\
1 & 3 & 3 & 1 \\
1 & 3 & 3 & 1 \\
1 & 3 & 3 & 1 \\
\end{tabular}
&
$S_{(4, 426)}$
&
\begin{tabular}{cccc}
1 & 1 & 1 & 1 \\
2 & 2 & 2 & 2 \\
1 & 1 & 1 & 1 \\
1 & 1 & 1 & 1 \\
\end{tabular}\\
\hline

$S_{(4, 427)}$
&
\begin{tabular}{cccc}
1 & 1 & 1 & 1\\
2 & 2 & 2 & 2 \\
1 & 1 & 1 & 1 \\
1 & 2 & 1 & 4 \\
\end{tabular}
&
$S_{(4, 428)}$
&
\begin{tabular}{cccc}
1 & 1 & 1 & 1 \\
2 & 2 & 2 & 2 \\
1 & 1 & 1 & 1 \\
1 & 2 & 3 & 4 \\
\end{tabular}
&
$S_{(4, 429)}$
&
\begin{tabular}{cccc}
1 & 2 & 1 & 1 \\
2 & 2 & 2 & 2 \\
1 & 2 & 1 & 1 \\
1 & 2 & 1 & 1 \\
\end{tabular}\\
\hline

$S_{(4, 430)}$
&
\begin{tabular}{cccc}
1 & 2 & 1 & 1\\
2 & 2 & 2 & 2 \\
1 & 2 & 1 & 1 \\
1 & 2 & 1 & 3 \\
\end{tabular}
&
$S_{(4, 431)}$
&
\begin{tabular}{cccc}
1 & 2 & 1 & 1 \\
2 & 2 & 2 & 2 \\
1 & 2 & 1 & 1 \\
1 & 2 & 1 & 4 \\
\end{tabular}
&
$S_{(4, 432)}$
&
\begin{tabular}{cccc}
1 & 2 & 1 & 1 \\
2 & 2 & 2 & 2 \\
1 & 2 & 1 & 3 \\
1 & 2 & 1 & 4 \\
\end{tabular}\\
\hline

$S_{(4, 433)}$
&
\begin{tabular}{cccc}
1 & 2 & 1 & 4\\
2 & 2 & 2 & 2 \\
1 & 2 & 1 & 4 \\
1 & 2 & 1 & 4 \\
\end{tabular}
&
$S_{(4, 434)}$
&
\begin{tabular}{cccc}
1 & 2 & 1 & 1 \\
2 & 2 & 2 & 2 \\
1 & 2 & 1 & 1 \\
1 & 2 & 3 & 4 \\
\end{tabular}
&
$S_{(4, 435)}$
&
\begin{tabular}{cccc}
1 & 2 & 1 & 1 \\
2 & 2 & 2 & 2 \\
1 & 2 & 1 & 3 \\
1 & 2 & 3 & 4 \\
\end{tabular}\\
\hline

$S_{(4, 436)}$
&
\begin{tabular}{cccc}
1 & 2 & 1 & 1\\
2 & 2 & 2 & 2 \\
1 & 2 & 3 & 1 \\
1 & 2 & 1 & 4 \\
\end{tabular}
&
$S_{(4, 437)}$
&
\begin{tabular}{cccc}
1 & 2 & 1 & 4 \\
2 & 2 & 2 & 2 \\
1 & 2 & 3 & 4 \\
1 & 2 & 1 & 4 \\
\end{tabular}
&
$S_{(4, 438)}$
&
\begin{tabular}{cccc}
1 & 2 & 1 & 1 \\
2 & 2 & 2 & 2 \\
1 & 2 & 3 & 4 \\
1 & 2 & 4 & 3 \\
\end{tabular}\\
\hline

$S_{(4, 439)}$
&
\begin{tabular}{cccc}
1 & 2 & 3 & 4\\
2 & 2 & 2 & 2 \\
1 & 2 & 3 & 4 \\
1 & 2 & 3 & 4 \\
\end{tabular}
&
$S_{(4, 440)}$
&
\begin{tabular}{cccc}
1 & 1 & 1 & 1 \\
2 & 2 & 2 & 2 \\
1 & 1 & 1 & 1 \\
2 & 2 & 2 & 2 \\
\end{tabular}
&
$S_{(4, 441)}$
&
\begin{tabular}{cccc}
1 & 2 & 1 & 2 \\
2 & 2 & 2 & 2 \\
1 & 2 & 1 & 2 \\
2 & 2 & 2 & 2 \\
\end{tabular}\\
\hline

$S_{(4, 442)}$
&
\begin{tabular}{cccc}
1 & 2 & 1 & 4\\
2 & 2 & 2 & 2 \\
1 & 2 & 1 & 4 \\
2 & 2 & 2 & 2 \\
\end{tabular}
&
$S_{(4, 443)}$
&
\begin{tabular}{cccc}
1 & 2 & 3 & 4 \\
2 & 2 & 2 & 2 \\
1 & 2 & 3 & 4 \\
2 & 2 & 2 & 2 \\
\end{tabular}
&
$S_{(4, 444)}$
&
\begin{tabular}{cccc}
1 & 1 & 1 & 1 \\
2 & 2 & 2 & 2 \\
1 & 1 & 1 & 1 \\
4 & 4 & 4 & 4 \\
\end{tabular}\\
\hline

$S_{(4, 445)}$
&
\begin{tabular}{cccc}
1 & 2 & 1 & 1\\
2 & 2 & 2 & 2 \\
1 & 2 & 1 & 1 \\
4 & 2 & 4 & 4 \\
\end{tabular}
&
$S_{(4, 446)}$
&
\begin{tabular}{cccc}
1 & 2 & 1 & 2 \\
2 & 2 & 2 & 2 \\
1 & 2 & 1 & 2 \\
4 & 2 & 4 & 2 \\
\end{tabular}
&
$S_{(4, 447)}$
&
\begin{tabular}{cccc}
1 & 2 & 1 & 4 \\
2 & 2 & 2 & 2 \\
1 & 2 & 1 & 4 \\
4 & 2 & 4 & 2 \\
\end{tabular}\\
\hline

$S_{(4, 448)}$
&
\begin{tabular}{cccc}
1 & 2 & 1 & 4 \\
2 & 2 & 2 & 2 \\
1 & 2 & 1 & 4 \\
4 & 2 & 4 & 4 \\
\end{tabular}
&
$S_{(4, 449)}$
&
\begin{tabular}{cccc}
1 & 2 & 1 & 1 \\
2 & 2 & 2 & 2 \\
1 & 2 & 3 & 1 \\
4 & 2 & 4 & 4 \\
\end{tabular}
&
$S_{(4, 450)}$
&
\begin{tabular}{cccc}
1 & 2 & 1 & 4 \\
2 & 2 & 2 & 2 \\
1 & 2 & 3 & 4 \\
4 & 2 & 4 & 2 \\
\end{tabular}\\
\hline

$S_{(4, 451)}$
&
\begin{tabular}{cccc}
1 & 2 & 1 & 4 \\
2 & 2 & 2 & 2 \\
1 & 2 & 3 & 4 \\
4 & 2 & 4 & 4 \\
\end{tabular}
&
$S_{(4, 452)}$
&
\begin{tabular}{cccc}
1 & 4 & 1 & 4 \\
2 & 2 & 2 & 2 \\
1 & 2 & 3 & 4 \\
4 & 4 & 4 & 4 \\
\end{tabular}
&
$S_{(4, 453)}$
&
\begin{tabular}{cccc}
1 & 4 & 1 & 4 \\
2 & 2 & 2 & 2 \\
1 & 4 & 1 & 4 \\
4 & 4 & 4 & 4 \\
\end{tabular}\\
\hline

$S_{(4, 454)}$
&
\begin{tabular}{cccc}
1 & 1 & 1 & 1\\
2 & 2 & 2 & 2 \\
3 & 3 & 3 & 3 \\
4 & 4 & 4 & 4 \\
\end{tabular}
&
$S_{(4, 455)}$
&
\begin{tabular}{cccc}
1 & 2 & 1 & 1\\
2 & 2 & 2 & 2 \\
3 & 2 & 3 & 3 \\
4 & 2 & 4 & 4 \\
\end{tabular}
&
$S_{(4, 456)}$
&
\begin{tabular}{cccc}
1 & 2 & 1 & 2\\
2 & 2 & 2 & 2 \\
3 & 2 & 3 & 2 \\
4 & 2 & 4 & 2 \\
\end{tabular}\\
\hline

\end{tabular}

\end{table}

\newpage

\begin{table}[ht]
\begin{tabular}{cccccc}
\hline
$S_{(4, 457)}$
&
\begin{tabular}{cccc}
1 & 2 & 3 & 4\\
2 & 2 & 2 & 2 \\
3 & 2 & 3 & 2 \\
4 & 2 & 2 & 4 \\
\end{tabular}
&
$S_{(4, 458)}$
&
\begin{tabular}{cccc}
1 & 3 & 3 & 1\\
2 & 2 & 2 & 2 \\
3 & 3 & 3 & 3 \\
4 & 2 & 2 & 4 \\
\end{tabular}
&
$S_{(4, 459)}$
&
\begin{tabular}{cccc}
1 & 1 & 1 & 1 \\
3 & 2 & 3 & 3 \\
3 & 3 & 3 & 3 \\
1 & 1 & 1 & 1 \\
\end{tabular}\\
\hline

$S_{(4, 460)}$
&
\begin{tabular}{cccc}
1 & 1 & 1 & 1 \\
3 & 2 & 3 & 2 \\
3 & 3 & 3 & 3 \\
1 & 4 & 1 & 4 \\
\end{tabular}
&
$S_{(4, 461)}$
&
\begin{tabular}{cccc}
1 & 1 & 1 & 1 \\
3 & 2 & 3 & 3 \\
3 & 3 & 3 & 3 \\
1 & 4 & 1 & 1 \\
\end{tabular}
&
$S_{(4, 462)}$
&
\begin{tabular}{cccc}
1 & 1 & 1 & 1 \\
3 & 3 & 3 & 3 \\
3 & 3 & 3 & 3 \\
1 & 1 & 1 & 1 \\
\end{tabular}\\
\hline
$S_{(4, 463)}$
&
\begin{tabular}{cccc}
1 & 2 & 3 & 1 \\
3 & 2 & 3 & 2 \\
3 & 2 & 3 & 3 \\
1 & 2 & 3 & 4 \\
\end{tabular}
&
$S_{(4, 464)}$
&
\begin{tabular}{cccc}
1 & 2 & 3 & 1 \\
3 & 2 & 3 & 3 \\
3 & 2 & 3 & 3 \\
1 & 2 & 3 & 1 \\
\end{tabular}
&
$S_{(4, 465)}$
&
\begin{tabular}{cccc}
1 & 2 & 3 & 4 \\
3 & 2 & 3 & 2 \\
3 & 2 & 3 & 2 \\
1 & 2 & 3 & 4 \\
\end{tabular}\\
\hline

$S_{(4, 466)}$
&
\begin{tabular}{cccc}
1 & 3 & 3 & 1\\
3 & 2 & 3 & 2 \\
3 & 3 & 3 & 3 \\
1 & 2 & 3 & 4 \\
\end{tabular}
&
$S_{(4, 467)}$
&
\begin{tabular}{cccc}
1 & 3 & 3 & 1\\
3 & 2 & 3 & 3 \\
3 & 3 & 3 & 3 \\
1 & 3 & 3 & 1 \\
\end{tabular}
&
$S_{(4, 468)}$
&
\begin{tabular}{cccc}
1 & 3 & 3 & 1\\
3 & 3 & 3 & 3 \\
3 & 3 & 3 & 3 \\
1 & 3 & 3 & 1 \\
\end{tabular}\\
\hline

$S_{(4, 469)}$
&
\begin{tabular}{cccc}
1 & 4 & 1 & 4\\
3 & 2 & 3 & 2 \\
3 & 2 & 3 & 2 \\
1 & 4 & 1 & 4 \\
\end{tabular}
&
$S_{(4, 470)}$
&
\begin{tabular}{cccc}
2 & 2 & 2 & 2 \\
2 & 2 & 2 & 2 \\
2 & 2 & 2 & 2 \\
2 & 2 & 2 & 2 \\
\end{tabular}
&
$S_{(4, 471)}$
&
\begin{tabular}{cccc}
3 & 2 & 2 & 3 \\
2 & 2 & 2 & 2 \\
2 & 2 & 2 & 2 \\
3 & 2 & 2 & 3 \\
\end{tabular}\\
\hline

$S_{(4, 472)}$
&
\begin{tabular}{cccc}
3 & 2 & 3 & 2 \\
2 & 2 & 2 & 2 \\
3 & 2 & 3 & 2 \\
2 & 2 & 2 & 2 \\
\end{tabular}
&
$S_{(4, 473)}$
&
\begin{tabular}{cccc}
3 & 3 & 3 & 3 \\
2 & 2 & 2 & 2 \\
3 & 3 & 3 & 3 \\
2 & 2 & 2 & 2 \\
\end{tabular}
&
$S_{(4, 474)}$
&
\begin{tabular}{cccc}
3 & 2 & 3 & 3 \\
2 & 2 & 2 & 2 \\
3 & 2 & 3 & 3 \\
3 & 2 & 3 & 3 \\
\end{tabular}\\
\hline

$S_{(4, 475)}$
&
\begin{tabular}{cccc}
3 & 3 & 3 & 3 \\
2 & 2 & 2 & 2 \\
3 & 3 & 3 & 3 \\
3 & 3 & 3 & 3 \\
\end{tabular}
&
$S_{(4, 476)}$
&
\begin{tabular}{cccc}
3 & 2 & 3 & 2 \\
3 & 2 & 3 & 2 \\
3 & 2 & 3 & 2 \\
3 & 2 & 3 & 2 \\
\end{tabular}
&
$S_{(4, 477)}$
&
\begin{tabular}{cccc}
3 & 2 & 3 & 3 \\
3 & 2 & 3 & 3 \\
3 & 2 & 3 & 3 \\
3 & 2 & 3 & 3 \\
\end{tabular}\\
\hline

$S_{(4, 478)}$
&
\begin{tabular}{cccc}
3 & 3 & 3 & 3 \\
3 & 2 & 3 & 2 \\
3 & 3 & 3 & 3 \\
3 & 2 & 3 & 2 \\
\end{tabular}
&
$S_{(4, 479)}$
&
\begin{tabular}{cccc}
3 & 3 & 3 & 3 \\
3 & 2 & 3 & 3 \\
3 & 3 & 3 & 3 \\
3 & 3 & 3 & 3 \\
\end{tabular}
&
$S_{(4, 480)}$
&
\begin{tabular}{cccc}
3 & 3 & 3 & 3 \\
3 & 3 & 3 & 3 \\
3 & 3 & 3 & 3 \\
3 & 3 & 3 & 3 \\
\end{tabular}\\
\hline
\end{tabular}
\end{table}

\begin{thm}\label{main}
$S_{(4, 435)}$ is the only nonfinitely based algebra in $S_{(4, k)}$, $388\leq k \leq 480$.
\end{thm}

In the following we shall introduce some notations that have not been used in \cite{rlzc}.
Let $\bu$ be an ai-semiring term such that $\bu=\bu_1+\bu_2+\cdots+\bu_n$,
where $\bu_i \in X^+$, $1 \leq i \leq n$.
Let $\bq$ be a nonempty word, and let $k$ be a positive integer. Then
\begin{itemize}
\item $L_{\geq k}(\bu)$ denotes the set $\{\bu_i \in \bu \mid \ell(\bu_i)\geq k\}$;

\item $L_{\leq k}(\bu)$ denotes the set $\{\bu_i \in \bu \mid \ell(\bu_i)\leq k\}$;

\item $L_k(\bu)$ denotes the set $\{\bu_i \in \bu \mid \ell(\bu_i)= k\}$;

\item $H_{\bq}(\bu)$ denotes the set $\{\bu_i \in \bu \mid h(\bu_i)=h(\bq)\}$;

\item $D_{\bq}(\bu)$ denotes the set $\{\bu_i \in \bu \mid c(\bu_i)\subseteq c(\bq)\}$;

\item $M_{1}(\bq)$ denotes $\{x \in c(\bq) \mid m(x, \bq)=1\}$.
\end{itemize}

Let $S$ be an ai-semiring. If we adjoin an extra element $0$ to the set $S$ and define
\[
(\forall a\in S\cup \{0\}) \quad a+0=a,~ a0=0a=0,
\]
then $S\cup \{0\}$ becomes an ai-semiring and is denoted by $S^0$.
The following result, which is due to \cite[Proposition 1.5]{wrz},
explores the relationship between the equational theories of $S^0$ and $S$.

\begin{lem}\label{lem001}
Let $\bu\approx \bu+\bq$ be an ai-semiring identity such that
$\bu=\bu_1+\bu_2+\cdots+\bu_n$, where $\bu_i, \bq \in X^+$, $1\leq i \leq n$.
Then $\bu\approx \bu+\bq$ is satisfied by ${S}^0$ if and only if
$D_\bq(\bu)\not=\emptyset$ and $D_\bq(\bu) \approx D_\bq(\bu)+\bq$ is satisfied by $S$.
\end{lem}

\section{The finite basis problem for specific 4-element ai-semirings}
In this section we first answer the finite basis problem for some $4$-element ai-semirings by some known results,
and then provide an equational basis for $S_{(4, 471)}$.

\begin{pro}
The ai-semiring $S_{(4, 435)}$ is nonfinitely based.
\end{pro}
\begin{proof}
It is easy to see that $S_{(4, 435)}$ is isomorphism to the semiring $S_7^0$.
By \cite[Corollary 2.4]{wrz} we immediately deduce that $S_{(4, 435)}$ is nonfinitely based.
\end{proof}

\begin{pro}\label{pro38801}
The following ai-semirings are finitely based: $S_{(4, 388)}$, $S_{(4, 415)}$, $S_{(4, 425)}$, $S_{(4, 462)}$,
$S_{(4, 468)}$, $S_{(4, 470)}$, $S_{(4, 472)}$, $S_{(4, 473)}$,
$S_{(4, 476)}$, $S_{(4, 478)}$ and $S_{(4, 480)}$.
\end{pro}
\begin{proof}
It is easy to check that every semiring in Proposition $\ref{pro38801}$
satisfies an equational basis (see \cite[Theorem 2.1]{sr})
of the variety generated by all ai-semirings of order two.
By the main result of \cite{sr} we have that these algebras are all finitely based.
\end{proof}

\begin{pro}\label{pro40901}
The following ai-semirings are finitely based: $S_{(4, 409)}$, $S_{(4, 421)}$, $S_{(4, 422)}$,
$S_{(4, 436)}$, $S_{(4, 437)}$, $S_{(4, 438)}$, $S_{(4, 439)}$,
$S_{(4, 449)}$, $S_{(4, 451)}$, $S_{(4, 452)}$, $S_{(4, 454)}$, $S_{(4, 455)}$,
$S_{(4, 457)}$, $S_{(4, 458)}$, $S_{(4, 460)}$,
$S_{(4, 463)}$, $S_{(4, 465)}$,
$S_{(4, 466)}$ and $S_{(4, 469)}$.
\end{pro}
\begin{proof}
It is easy to see that every semiring in Proposition $\ref{pro40901}$ except $S_{(4, 438)}$ satisfies $x^2 \approx x$. From
the main result of \cite{gpz, pas05} we deduce that these semirings are all finitely based.
Since $S_{(4, 438)}$ satisfies $x^3 \approx x$,
it follows from the main result of \cite{rzw} that $S_{(4, 438)}$ is finitely based.
\end{proof}
\begin{pro}
The following ai-semirings are finitely based: $S_{(4, 389)}$, $S_{(4, 391)}$, $S_{(4, 399)}$,
$S_{(4, 402)}$, $S_{(4, 403)}$, $S_{(4, 406)}$, $S_{(4, 407)}$, $S_{(4, 408)}$, $S_{(4, 410)}$,
$S_{(4, 416)}$, $S_{(4, 420)}$, $S_{(4, 426)}$, $S_{(4, 429)}$, $S_{(4, 444)}$ and $S_{(4, 448)}$.
\end{pro}
\begin{proof}
From \cite[Corollary 5.1]{jrz} we know that every ai-semiring of order three is finitely based except $S_7$.
It is a routine matter that $S_{(4,389)}$ is isomorphic to a subdirect product of two copies of $S_{60}$.
This implies that $\mathsf{V}(S_{(4, 389)}) =\mathsf{V}(S_{60})$
and so $S_{(4,389)}$ is finitely based. One can use the same approach to prove that
\[
\mathsf{V}(S_{(4, 391)}) = \mathsf{V}(S_{57}),
\mathsf{V} (S_{(4, 403)}) = \mathsf{V}(S_{54}),
\mathsf{V}(S_{(4, 407)}) = \mathsf{V}(S_{53}),
\]
\[
\mathsf{V}(S_{(4, 416)}) = \mathsf{V}(S_{56}),
\mathsf{V}(S_{(4, 426)}) = \mathsf{V}(S_{58}),
\mathsf{V}(S_{(4, 429)}) = \mathsf{V}(S_{55}).
\]
So $S_{(4, 391)}$, $S_{(4, 403)}$, $S_{(4, 407)}$, $S_{(4, 416)}$, $S_{(4, 426)}$ and $S_{(4, 429)}$
are all finitely based.

On the other hand,
it is easy to verify that $S_{(4, 399)}$ is isomorphic to a subdirect product of $S_{30}$ and $S_{60}$,
and that $S_{30}$ satisfies a finite equational basis (see \cite[Proposition 12]{zrc}) of $S_{60}$.
It follows that $\mathsf{V}(S_{(4, 399)}) = \mathsf{V}(S_{60})$
and so $S_{(4,399)}$ is finitely based. One can use the same approach to prove that
\[
\mathsf{V}(S_{(4, 402)}) = \mathsf{V}(S_{57}),
\mathsf{V}(S_{(4, 406)}) = \mathsf{V}(S_{54}),
\mathsf{V}(S_{(4, 408)}) = \mathsf{V}(S_{53}),
\]
\[
\mathsf{V}(S_{(4, 410)}) = \mathsf{V}(S_{59}),
\mathsf{V}(S_{(4, 420)}) = \mathsf{V}(S_{56}),
\mathsf{V}(S_{(4, 444)}) = \mathsf{V}(S_{58}),
\mathsf{V}(S_{(4, 448)}) = \mathsf{V}(S_{55}).
\]
Thus $S_{(4, 402)}$, $S_{(4, 406)}$, $S_{(4, 408)}$, $S_{(4, 410)}$, $S_{(4, 420)}$,
$S_{(4, 444)}$ and $S_{(4, 448)}$ are all finitely based.
\end{proof}

\begin{pro}\label{pro47101}
$\mathsf{V}(S_{(4, 471)})$ is the ai-semiring variety defined by the identities
\begin{align}
xy &\approx yx; \label{47101}\\
x^2 &\approx x^2+xy; \label{47102}\\
x_1 & \approx x_1+x_2x_3x_4. \label{47104}
\end{align}
\end{pro}
\begin{proof}
It is easy to check that $S_{(4, 471)}$ satisfies the identities (\ref{47101})--(\ref{47104}).
In the remainder we need only show that every ai-semiring identity of $S_{(4, 471)}$
can be derived by (\ref{47101})--(\ref{47104}) and the identities defining $\mathbf{AI}$.
Let $\bu \approx \bu+\bq$ be such a nontrivial identity,
where $\bu=\bu_1+\bu_2+\cdots+\bu_n$ and $\bu_i, \bq\in X^+$, $1 \leq i \leq n$.
Since $N_2$ is isomorphic to $\{2, 3\}$,
it follows that $N_2$ satisfies $\bu \approx \bu+\bq$ and so $\ell(\bq)\geq 2$.
If $\ell(\bq)\geq 3$, then by (\ref{47104}) we immediately deduce
$\bu\approx \bu+\bq$.

Now let $\ell(\bq)=2$.
We shall show that $\ell(\bu_i)=2$ for some $\bu_i \in D_\bq(\bu)$.
Suppose by way of contradiction that $\ell(\bu_i)\not=2$ for all $\bu_i \in D_\bq(\bu)$.
Consider the substitution $\varphi: P_f(X^+) \to S_{(4, 471)}$ defined by
$\varphi(x)=4$ for every $x\in c(\bq)$ and $\varphi(x)=2$ otherwise.
Then $\varphi(\bu)=4$ or $2$, $\varphi(\bq)=3$ and so $\varphi(\bu)\neq \varphi(\bu+\bq)$, a contradiction.
So $\ell(\bu_i)=2$ for some $\bu_i \in D_\bq(\bu)$.
We may write that $\bq=xy$. Then $\bu_i=xy$, $yx$, $x^2$ or $y^2$.
Since the identity (\ref{47101}) holds in $S_{(4, 471)}$
and $\bu \approx \bu+\bq$ is nontrivial, we only consider the case that
$\bu_i=x^2$.
Indeed, we have
\[
\bu \approx \bu+\bu_i\approx \bu+x^2\stackrel{(\ref{47102})}\approx \bu+x^2+xy\approx \bu+x^2+\bq.
\]
This proves the identity $\bu \approx \bu+\bq$.
\end{proof}

\section{Equational basis of some $4$-element ai-semirings that relate to $S_{57}$}
In this section we present equational basis for some $4$-element ai-semirings that relate to $S_{57}$.
As a first step, we give some information about the identities of $S_{57}$.

\begin{lem}\label{lem5701}
Let $\bu\approx \bu+\bq$ be a nontrivial ai-semiring identity such that
$\bu=\bu_1+\bu_2+\cdots+\bu_n$ and $\bu_i, \bq \in X^+$, $1\leq i \leq n$.
If $\bu\approx \bu+\bq$ holds in $S_{57}$,
then $\ell(\bu_k)\geq 2$ for some $\bu_k \in \bu$,
$c(p(\bq))\subseteq c(p(\bu))$ and $t(\bq)\in c(\bu)$.
\end{lem}
\begin{proof}
Assume that $S_{57}$ satisfies $\bu\approx \bu+\bq$.
Since $T_2$ is isomorphic to $\{1, 3\}$,
we have that $T_2$ satisfies $\bu \approx \bu+\bq$
and so $\ell(\bu_k)\geq 2$ for some $\bu_k \in \bu$.
Since $M_2$ is isomorphic to $\{2, 3\}$,
it follows that $M_2$ satisfies $\bu \approx \bu+\bq$ and so $c(\bq)\subseteq c(\bu)$.
Suppose by way of contradiction that $c(p(\bq))\nsubseteq c(p(\bu))$.
Then there exists $x\in c(p(\bq))$ such that $x\notin c(p(\bu))$.
Let $\varphi: P_f(X^+)\to S_{57}$ be a homomorphism defined by
$\varphi(x)=1$ and $\varphi(y)=2$ if $y\neq x$.
Then $\varphi(\bu)=1$ and $\varphi(\bq)=3$, and so $\varphi(\bu)\neq \varphi(\bu+\bq)$, a contradiction.
Thus $c(p(\bq))\subseteq c(p(\bu))$ as required.
\end{proof}

\begin{pro}\label{pro42401}
$\mathsf{V}(S_{(4, 424)})$ is the ai-semiring variety defined by the identities
\begin{align}
x^2y&\approx xy; \label{42401}\\
xyz &\approx yxz; \label{42402}\\
xy &\approx xy+y; \label{42404}\\
x+yz & \approx yx+yz.\label{42403}
\end{align}
\end{pro}
\begin{proof}
It is easily verified that $S_{(4, 424)}$ satisfies the identities
(\ref{42401})--(\ref{42403}).
In the remainder we need only show that every ai-semiring identity of $S_{(4, 424)}$
can be derived by (\ref{42401})--(\ref{42403}) and the identities defining $\mathbf{AI}$.
Let $\bu \approx \bu+\bq$ be such a nontrivial identity,
where $\bu=\bu_1+\bu_2+\cdots+\bu_n$ and $\bu_i, \bq\in X^+$, $1 \leq i \leq n$.
It is easy to see that $R_2$ is isomorphic to $\{1, 3\}$ and so $R_2$ satisfies  $\bu \approx \bu+\bq$.
This implies that $t(\bq)\in t(\bu)$.
On the other hand, we have that there exists a congruence $\rho=\{(1,3),(3,1)\}\cup \vartriangle$
 such that $S_{(4, 424)}/\rho$ is isomorphic to $S_{57}$ and so $S_{57}$ satisfies $\bu \approx \bu+\bq$.
By Lemma \ref{lem5701} it follows that
$\ell(\bu_i)\geq 2$ for some $\bu_i \in \bu$,
$c(p(\bq))\subseteq c(p(\bu))$ and $t(\bq)\in c(\bu)$.

Suppose that
$
c(p(\bu))=\{x_1, x_2, \ldots, x_m\}
$
and that
$
t(\bu)=\{y_1, y_2, \ldots, y_n\}.
$
Then
\[
c(p(\bq)) \subseteq \{x_1, x_2, \ldots, x_m\}, t(\bq)\in \{y_1, y_2, \ldots, y_n\}.
\]
Now we have
\begin{align*}
\bu
&\approx x_1^2x_2^2\cdots x_m^2(y_1+y_2+\cdots+y_n)&&(\text{by}~(\ref{42401}), (\ref{42402}), (\ref{42403}))\\
&\approx x_1^2x_2^2\cdots x_m^2(y_1+y_2+\cdots+y_n)+x_1^2x_2^2\cdots x_m^2t(\bq).
\end{align*}
This derives the identity
\begin{equation}\label{id04081701}
\bu\approx \bu+x_1^2x_2^2\cdots x_m^2t(\bq).
\end{equation}
Furthermore, we have
\begin{align*}
\bu
&\approx \bu+x_1^2 x_2^2\cdots x_m^2t(\bq)&&(\text{by}~(\ref{id04081701}))\\
&\approx \bu+\bp p(\bq)t(\bq) &&(\text{by}~(\ref{42401}), (\ref{42402})) \\
&\approx \bu+\bp \bq \\
&\approx  \bu+\bp\bq+\bq. &&(\text{by}~(\ref{42404}))
\end{align*}
This implies the identity $\bu \approx \bu+\bq$.
\end{proof}

\begin{remark}
It is a routine matter to verify that $S_{(4, 424)}$ is isomorphic to a subdirect product of $S_{57}$ and $S_{29}$.
So $\mathsf{V}(S_{(4, 424)}) = \mathsf{V}(S_{57}, S_{29})$.
On the other hand, it is easy to check that both $M_2$ and $R_2$ can be embedded into $S_{29}$.
Also, $S_{29}$ satisfies an equational basis of $\mathsf V(R_2, M_2)$ that can be found in \cite{sr}.
It follows that $\mathsf V(S_{29})=\mathsf V(R_2, M_2)$, and so $\mathsf V(S_{(4, 424)}) = \mathsf V(S_{57}, R_2, M_2)$.
Since $M_2$ can be embedded into $S_{57}$, we now have
\[
\mathsf V(S_{(4, 424)}) = \mathsf V(S_{57}, R_2).
\]
\end{remark}

Notice that $S_{(4, 461)}$ and $S_{(4, 424)}$  have dual multiplications. By Proposition \ref{pro42401} we immediately deduce
\begin{cor}
The ai-semiring $S_{(4, 461)}$ is finitely based.
\end{cor}

\begin{cor}
The ai-semirings $S_{(4, 395)}$ and $S_{(4, 404)}$ are both finitely based.
\end{cor}
\begin{proof}
Firstly, one can observe that $S_{(4, 395)}$ and $S_{(4, 404)}$  have dual multiplications.
Next, it is easy to verify that $S_{(4, 395)}$ satisfies the identities (\ref{42401})--(\ref{42403}).
On the other hand, $S_{(4, 424)}$ can be embedded into the direct product of two copies of $S_{(4, 395)}$.
So $\mathsf{V}(S_{(4, 424)})=\mathsf{V}(S_{(4, 395)})$.
By Proposition \ref{pro42401} we immediately deduce that $S_{(4, 395)}$ and $S_{(4, 404)}$ are both finitely based.
\end{proof}

\begin{pro}\label{pro40001}
$\mathsf{V}(S_{(4, 400)})$ is the ai-semiring variety defined by the identities
\begin{align}
xyz& \approx xyz+y;\label{40004}\\
xyz& \approx xy+yz+xz;\label{40005}\\
x_1x_2+x_3x_4 &\approx x_1x_2+x_3x_4+x_1x_4, \label{40006}
\end{align}
where $x$ and $z$ may be empty in $(\ref{40004})$.
\end{pro}
\begin{proof}
It is easy to verify that $S_{(4, 400)}$ satisfies the identities (\ref{40004})--(\ref{40006}).
In the remainder it is enough to show that every ai-semiring identity of $S_{(4, 400)}$
is derivable from (\ref{40004})--(\ref{40006}) and the identities defining $\mathbf{AI}$.
Let $\bu \approx \bu+\bq$ be such a nontrivial identity, $\bu=\bu_1+\bu_2+\cdots+\bu_n$, where $\bu_i, \bq\in X^+$, $1 \leq i \leq n$.
It is a routine matter to verify that $S_{(4, 400)}$ is isomorphic to
a subdirect product of $S_{54}$ and $S_{57}$.
So both $S_{54}$ and $S_{57}$ satisfy $\bu \approx \bu+\bq$.

\textbf{Case 1.} $\ell(\bq)=1$.
By Lemma \ref{lem5701} we obtain that $c(\bq)\subseteq c(\bu)$ and so $c(\bq)\subseteq c(\bu_j)$ for some $\bu_j \in \bu$.
Thus $\bu_j=\bp_1\bq\bp_2$ for some $\bp_1, \bp_2\in X^*$.
Now we have
\[
\bu \approx \bu+\bu_j \approx \bu+\bp_1\bq\bp_2\stackrel{(\ref{40004})}\approx \bu+\bp_1\bq\bp_2+\bq.
\]
This implies the identity $\bu\approx\bu+\bq$.

\textbf{Case 2.} $\ell(\bq)\geq 2$. Let $\bq=x_1x_2\cdots x_n$, where $n\geq 2$. Then
by the identity (\ref{40005}) we deduce the identity
\[
\bq \approx\sum\limits_{1\leq i\textless j\leq n}x_ix_j.
\]
So it is enough to consider the case that $\bq=xy$.
Since $S_{54}$ and $S_{57}$ have dual multiplications,
it follows from Lemma \ref{lem5701} and its dual that
$x\in c(p(\bu))$ and $y\in c(s(\bu))$.
So there exist $\bp_1, \bp_2, \bp_3, \bp_4$ such that $\bu_i=\bp_1x\bp_2$ and $\bu_j=\bp_3y\bp_4$,
where $\bp_2$ and $\bp_3$ are both nonempty.
We now deduce
\begin{align*}
\bu
&\approx \bu+\bu_i+\bu_j\\
&\approx \bu+\bp_1x\bp_2+\bp_3y\bp_4\\
&\approx \bu+\bp_1x\bp_2+\bp_3y\bp_4+\bp_1xy\bp_4&&(\text{by}~(\ref{40006}))\\
&\approx \bu+\bp_1x\bp_2+\bp_3y\bp_4+\bp_1xy\bp_4+xy &&(\text{by}~(\ref{40004}))\\
&\approx \bu+\bp_1x\bp_2+\bp_3y\bp_4+\bp_1xy\bp_4+\bq.
\end{align*}
This derives the identity $\bu\approx\bu+\bq$.
\end{proof}

\begin{lem}\label{lem5301}
Let $\bu\approx \bu+\bq$ be a nontrivial ai-semiring identity such that
$\bu=\bu_1+\bu_2+\cdots+\bu_n$ and $\bu_i, \bq \in X^+$, $1\leq i \leq n$.
Suppose that $\bu\approx \bu+\bq$ holds in $S_{53}$. Then $\bu$ and $\bq$ satisfy the following conditions:
\begin{itemize}
\item[$(1)$] $\ell(\bu_i)\geq 2$ for some $\bu_i \in \bu$;

\item[$(2)$] $c(\bq)\subseteq c(\bu)$;

\item[$(3)$] Let $\bq=xy$.
If $x=y$, then $m(x, \bu_k)\geq 2$ for some $\bu_k\in \bu$;
if $x\neq y$, then there exists $\bu_k \in \bu$ such that either
$x, y\in c(\bu_k)$ or $m(x, \bu_k)\geq 2$ or $m(y, \bu_k)\geq 2$.
\end{itemize}
\end{lem}
\begin{proof}
Suppose that $S_{53}$ satisfies $\bu\approx \bu+\bq$.
Since $T_2$ is isomorphic to $\{1, 3\}$, it follows that $T_2$ satisfies $\bu \approx \bu+\bq$
and so $\ell(\bu_i)\geq 2$ for some $\bu_i \in \bu$.
It is easy to see that $M_2$ is isomorphic to $\{2, 3\}$ and so $M_2$ satisfies  $\bu \approx \bu+\bq$.
This implies that $c(\bq)\subseteq c(\bu)$.
Let $\bq=xy$. Consider the following two cases.

\textbf{Case 1.} $x=y$. We shall show that $m(x, \bu_k)\geq 2$ for some $\bu_k\in \bu$.
Suppose that this is not true.
Consider the semiring homomorphism $\varphi: P_f(X^+) \to S_{53}$ defined by
$\varphi(z)=1$ if $z=x$ and $\varphi(z)=2$ otherwise.
It is easy to see that $\varphi(\bu)=1$ and $\varphi(\bq)=3$.
So $\varphi(\bu)\neq \varphi(\bu+\bq)$, a contradiction.
Thus $m(x, \bu_k)\geq 2$ for some $\bu_k\in \bu$.

\textbf{Case 2.} $x\neq y$.
Suppose that for any $\bu_k \in \bu$,
$x \notin c(\bu_k)$ or $y \notin c(\bu_k)$, $m(x, \bu_k)\leq 1$ and $m(y, \bu_k)\leq 1$.
Let $\varphi: P_f(X^+) \to S_{53}$ be a semiring homomorphism such that
$\varphi(z)=1$ if $z=x$ or $y$, $\varphi(z)=2$ otherwise.
Then $\varphi(\bu)=1$ and $\varphi(\bq)=3$.
This implies that $\varphi(\bu)\neq \varphi(\bu+\bq)$,
a contradiction.
So there exists $\bu_k \in \bu$ such that
either $x, y\in c(\bu_k)$ or $m(x, \bu_k)\geq 2$ or $m(y, \bu_k)\geq 2$.
\end{proof}

\begin{pro}\label{pro40101}
$\mathsf{V}(S_{(4, 401)})$ is the ai-semiring variety defined by the identities
\begin{align}
xyz &\approx yxz; \label{40101}\\
xyz &\approx xyz+y; \label{40103}\\
xyz& \approx xy+yz+xz;\label{40105}\\
x^2+y & \approx x^2y;\label{40104}\\
x_1x_2x_3+x_3x_4&\approx x_1x_2x_3+x_3x_4+x_3x_2, \label{40106}
\end{align}
where $x$ and $z$ may be empty in $(\ref{40103})$, $x_1$ may be empty in $(\ref{40106})$.
\end{pro}
\begin{proof}
It is easy to check that $S_{(4, 401)}$ satisfies the identities (\ref{40101})--(\ref{40106}).
In the remainder it is enough to show that every ai-semiring identity of $S_{(4, 401)}$
is derivable from (\ref{40101})--(\ref{40106}) and the identities defining $\mathbf{AI}$.
Let $\bu \approx \bu+\bq$ be such a nontrivial identity, where
$\bu=\bu_1+\bu_2+\cdots+\bu_n$ and $\bu_i, \bq \in X^+$, $1 \leq i \leq n$.
It is a routine matter to verify that $S_{(4, 401)}$ is isomorphic to
a subdirect product of $S_{53}$ and $S_{57}$.
So both $S_{53}$ and $S_{57}$ satisfy $\bu \approx \bu+\bq$.

\textbf{Case 1.} $\ell(\bq)=1$.
By Lemma \ref{lem5701} we can obtain that $c(\bq)\subseteq c(\bu)$ and so $c(\bq)\subseteq c(\bu_j)$ for some $\bu_j \in \bu$.
Since $\bu \approx \bu+\bq$ is nontrivial, it follows that $\ell(\bu_j)\geq 2$.
By the identity (\ref{40101}) we deduce $\bu_j\approx \bq \bp$ or $\bu_j\approx \bp\bq$  for some nonempty word $\bp$.
Now we have
\[
\bu \approx \bu+\bu_j \stackrel{(\ref{40103})}\approx \bu+\bu_j+\bq.
\]

\textbf{Case 2.} $\ell(\bq)\geq 2$.
Let $\bq=x_1x_2\cdots x_n$, where $n\geq 2$. Then
by the identity (\ref{40105}) we deduce the identity
\[
\bq \approx\sum\limits_{1\leq i\textless j\leq n}x_ix_j.
\]
So we only need to consider the case that $\bq=xy$.

\textbf{Subcase 2.1.}  $x=y$. By Lemma \ref{lem5301}
there exists $\bu_k\in \bu$ such that $m(x, \bu_k)\geq 2$ and so by (\ref{40101}) we deduce
$\bu_k\approx\bp_1x^2\bp_2$ for some $\bp_1, \bp_2\in X^*$.
Furthermore, we have
\[
\bu \approx \bu+\bu_k \approx \bu+\bp_1x^2\bp_2\stackrel{(\ref{40103})}\approx \bu+\bp_1x^2\bp_2+x^2\approx \bu+\bp_1x^2\bp_2+\bq.
\]
This implies the identity $\bu \approx \bu+\bq$.

\textbf{Subcase 2.2.} $x\neq y$. By Lemma \ref{lem5301}
there exists $\bu_k \in \bu$ such that
either $x, y\in c(\bu_k)$ or $m(x, \bu_k)\geq 2$ or $m(y, \bu_k)\geq 2$.
Firstly, suppose that $x, y\in c(\bu_k)$ for some $\bu_k\in \bu$.
Then by the identity (\ref{40101}) we deduce
$\bu_k\approx \bp_1 xy\bp_2$ or $\bu_k\approx \bp_3yx$ for some $\bp_1, \bp_2, \bp_3 \in X^*$.
If $\bu_k\approx \bp_1 xy\bp_2$ is obtained, then
\[
\bu
\approx \bu+\bu_k \approx \bu+\bp_1 xy\bp_2
\stackrel{(\ref{40103})}\approx  \bu+\bp_1 xy\bp_2+xy
\approx \bu+\bp_1 xy\bp_2+\bq.
\]
If $\bu_k\approx \bp_3yx$ is derived,
then by Lemma \ref{lem5701} there exists $\bu_j\in \bu$ such that $x\in c(p(\bu_j))$
and so the identity (\ref{40101}) implies $\bu_j\approx x\bp$ for some nonempty word $\bp$.
We now deduce
\begin{align*}
\bu
&\approx \bu+\bu_k+\bu_j\\
&\approx \bu+\bp_3yx+x\bp \\
&\approx \bu+\bp_3yx+x\bp+xy.&&(\text{by}~(\ref{40106}))
\end{align*}

Next, assume that there exists $\bu_k \in \bu$ such that $m(x, \bu_k)\geq 2$.
By a similar argument in Subcase 2.1 one can obtain the identity
\[
\bu\approx \bu+x^2.
\]
Since $c(\bq)\subseteq c(\bu)$, we have that $y\in c(\bu_j)$ for some $\bu_j \in \bu$.
Furthermore, we can deduce
\begin{align*}
\bu
&\approx \bu+x^2+\bu_j\\
&\approx \bu+x^2\bu_j&&(\text{by}~(\ref{40104}))\\
&\approx \bu+\bq_1xy\bq_2&&(\text{by}~(\ref{40101}))\\
&\approx \bu+\bq_1xy\bq_2+xy. &&(\text{by}~(\ref{40103}))
\end{align*}

Finally, assume that there exists $\bu_k \in \bu$ such that $m(y, \bu_k)\geq 2$.
By a similar argument in Subcase 2.1 one can deduce the identity
\[
\bu\approx \bu+y^2.
\]
Since $x\in c(p(\bq))$, it follows from Lemma \ref{lem5701} that
$x\in c(p(\bu_j))$ for some $\bu_j \in \bu$ and so $\bu_j=\bp_1x\bp_2$
for some $\bp_1\in X^*$ and $\bp_2\in X^+$. So we have
\begin{align*}
\bu
&\approx \bu+y^2+\bu_j\\
&\approx \bu+y^2\bu_j&&(\text{by}~(\ref{40104}))\\
&\approx \bu+y^2\bp_1x\bp_2\\
&\approx \bu+y\bp_1xy\bp_2&&(\text{by}~(\ref{40101}))\\
&\approx \bu+y\bp_1xy\bp_2+xy. &&(\text{by}~(\ref{40103}))
\end{align*}
This proves the identity $\bu \approx \bu+\bq$.
\end{proof}

\begin{cor}
The ai-semiring $S_{(4, 405)}$ is finitely based.
\begin{proof}
It is easy to see that $S_{(4, 405)}$ and $S_{(4, 401)}$  have dual multiplications.
By Proposition \ref{pro40101} we immediately deduce that $S_{(4, 405)}$ is finitely based.
\end{proof}
\end{cor}

\section{Equational basis of some $4$-element ai-semirings that relate to $S_{58}$}
In this section we provide equational basis for some $4$-element ai-semirings that relate to $S_{58}$.
We first give some information about the identities of $S_{58}$.
\begin{lem}\label{lem5801}
Let $\bu\approx \bu+\bq$ be a nontrivial ai-semiring identity such that $\bu=\bu_1+\bu_2+\cdots+\bu_n$
and $\bu_i, \bq \in X^+$, $1\leq i \leq n$.
Suppose that $\bu\approx \bu+\bq$ is satisfied by $S_{58}$. Then $\bu$ and $\bq$ satisfy the following conditions:
\begin{itemize}
\item[$(1)$] $L_{\geq 2}(\bu) \neq \emptyset$;

\item[$(2)$] $H_{\bq}(\bu) \neq \emptyset$;

\item[$(3)$] If $\ell(\bq)\geq 2$, then
$L_{\geq 2}(\bu)\cap H_{\bq}(\bu) \neq \emptyset$.
\end{itemize}
\end{lem}
\begin{proof}
Suppose that $S_{58}$ satisfies $\bu\approx \bu+\bq$.
It is easy to see that $T_2$ is isomorphic to $\{1, 3\}$ and so $T_2$ satisfies $\bu \approx \bu+\bq$.
Thus $L_{\geq 2}(\bu)$ is nonempty.
Since $L_2$ is isomorphic to $\{2, 3\}$,
it follows that $L_2$ satisfies $\bu \approx \bu+\bq$ and so $H_{\bq}(\bu)$ is nonempty.
Let $\ell(\bq)\geq 2$.
Suppose by way of contradiction that $L_{\geq 2}(\bu)\cap H_{\bq}(\bu)$ is empty.
Then $h(\bq) \in L_{1}(\bu)$.
Consider the substitution $\varphi: P_f(X^+) \to S_{58}$ defined by
$\varphi(x)=1$ if $x=h(\bq)$, and $\varphi(x)=2$ otherwise.
We can obtain that $\varphi(\bu)=1$, $\varphi(\bq)=3$ and so
$\varphi(\bu)\neq \varphi(\bu+\bq)$,
a contradiction.
Thus $L_{\geq 2}(\bu)\cap H_{\bq}(\bu)$ is nonempty.
\end{proof}

\begin{pro}\label{pro45301}
$\mathsf{V}(S_{(4, 453)})$ is the ai-semiring variety defined by the identities
\begin{align}
x^2y &\approx xy; \label{45301}\\
xy^2 &\approx xy; \label{45302}\\
xyz &\approx xzy; \label{45303}\\
x^2 &\approx x^2+x; \label{45304}\\
xy & \approx xy+xyz; \label{45305}\\
x+yz & \approx x+yz+yx. \label{45306}
\end{align}
\end{pro}
\begin{proof}
It is easy to check that $S_{(4, 453)}$ satisfies the identities (\ref{45301})--(\ref{45306}).
In the remainder we need only show that every ai-semiring identity of $S_{(4, 453)}$
is derivable from (\ref{45301})--(\ref{45306}) and the identities defining $\mathbf{AI}$.
Let $\bu \approx \bu+\bq$ be such a nontrivial identity,
where $\bu=\bu_1+\bu_2+\cdots+\bu_n$ and $\bu_i, \bq\in X^+$, $1 \leq i \leq n$.
Since $D_2$ is isomorphic to $\{1, 4\}$, it follows that $D_2$ satisfies $\bu \approx \bu+\bq$
and so $c(\bu_r)\subseteq c(\bq)$ for some $\bu_r \in \bu$.
On the other hand, we have that there exists a congruence $\rho=\{(1,4),(4,1)\}\cup \vartriangle$
 such that $S_{(4, 453)}/\rho$ is isomorphic to $S_{58}$ and so $S_{58}$ satisfies $\bu \approx \bu+\bq$.
By Lemma \ref{lem5801} we consider the following two cases.

\textbf{Case 1.} $\ell(\bq)=1$. Then $\bu_r=\bq^k$ for some $k\geq2$. So we have
\[
\bu \approx \bu+\bu_r \approx \bu+\bq^k\stackrel{(\ref{45304})}\approx \bu+\bq^k+\bq.
\]
This implies the identity $\bu \approx \bu+\bq$.

\textbf{Case 2.} $\ell(\bq)\geq 2$. By Lemma ${\ref{lem5801}}$ we obtain that
$L_{\geq 2}(\bu)\cap H_{\bq}(\bu)$ is nonempty
and so there exists $\bu_j \in \bu$ such that $\ell(\bu_j)\geq 2$ and $h(\bu_j)=h(\bq)$.
Now we have
\[
\bu\approx \bu+\bu_j \stackrel{(\ref{45305})}\approx \bu+\bu_j+\bu_j\bq\approx\bu+\bu_j+h(\bu_j)s(\bu_j)\bq
\approx\bu+\bu_j+h(\bq)s(\bu_j)\bq.
\]
This derives the identity
\[
\bu\approx \bu+h(\bq)s(\bu_j)\bq.
\]
Furthermore, we can deduce

\begin{align*}
\bu
&\approx \bu+\bu_r+h(\bq)s(\bu_j)\bq \\
&\approx  \bu+\bu_r+h(\bq)\bq s(\bu_j) &&(\text{by}~(\ref{45303}))\\
&\approx  \bu+\bu_r+h(\bq)\bq s(\bu_j)+h(\bq)\bq\bu_r  &&(\text{by}~(\ref{45306}))\\
&\approx \bu+\bu_r+h(\bq)\bq s(\bu_j)+\bq. &&(\text{by}~(\ref{45301}), (\ref{45302}), (\ref{45303}))
\end{align*}
This proves the identity $\bu \approx \bu+\bq$.
\end{proof}

\begin{remark}
It is a routine matter to verify that $S_{(4, 453)}$ is isomorphic to a subdirect product of $S_{58}$ and $S_{25}$.
So $\mathsf V(S_{(4, 453)}) = \mathsf V(S_{58}, S_{25})$. On the other hand, it is easy
to see that both $L_2$ and $D_2$ can be embedded into $S_{25}$. Also,
$S_{25}$ satisfies an equational basis of $\mathsf V(L_2, D_2)$ that can be found in \cite{sr}.
It follows that $\mathsf V(S_{25})=\mathsf V(L_2, D_2)$, and so
$\mathsf V(S_{(4, 453)}) = \mathsf V(S_{58}, L_2, D_2)$. Since $L_2$ can be embedded into $S_{58}$, we therefore obtain
\[
\mathsf{V}(S_{(4, 453)})=\mathsf{V}(S_{58}, D_2).
\]
\end{remark}

Notice that $S_{(4, 464)}$ is isomomorphism to the dual of $S_{(4, 453)}$. By Proposition \ref{pro45301} we immediately deduce
\begin{cor}
The ai-semiring $S_{(4, 464)}$ is finitely based.
\end{cor}

\begin{pro}\label{pro47501}
$\mathsf{V}(S_{(4, 475)})$ is the ai-semiring variety defined by the identity
\begin{align}
xy & \approx x^2. \label{47501}
\end{align}
\end{pro}
\begin{proof}
It is easy to check that $S_{(4, 475)}$ satisfies the identity (\ref{47501}).
In the remainder we shall show that every ai-semiring identity of $S_{(4, 475)}$
is derivable from $(\ref{47501})$ and the identities defining $\mathbf{AI}$.
Let $\bu \approx \bu+\bq$ be such a nontrivial identity, where
$\bu=\bu_1+\bu_2+\cdots+\bu_n$ and $\bu_i, \bq \in X^+$, $1 \leq i \leq n$.
Since $N_2$ is isomorphic to $\{1, 3\}$,
it follows that $N_2$ satisfies $\bu \approx \bu+\bq$ and so $\ell(\bq)\geq 2$.
On the other hand, we have that there exists a congruence $\rho=\{(1,3),(3,1)\}\cup \vartriangle$
 such that $S_{(4, 475)}/\rho$ is isomorphic to $S_{58}$ and so $S_{58}$ satisfies $\bu \approx \bu+\bq$.
By Lemma \ref{lem5801} it follows that there exists $\bu_j \in \bu$ such that
$\ell(\bu_j)\geq 2$ and $h(\bu_j)=h(\bq)$.
This implies that $\bu_j=h(\bq)s(\bu_j)$, where $s(\bu_j)$ is nonempty.
Now we have
\[
\bu \approx \bu+\bu_j \approx  \bu+h(\bq)s(\bu_j)\stackrel{(\ref{47501})}\approx \bu+{h(\bq)}^2
\stackrel{(\ref{47501})}\approx \bu+h(\bq)s(\bq)\approx \bu+\bq.
\]
This derives the identity $\bu \approx \bu+\bq$.
\end{proof}

\begin{remark}
It is a routine matter to verify that $S_{(4, 475)}$ is isomorphic to a subdirect product of $S_{58}$ and $S_{24}$. So $\mathsf V(S_{(4, 475)}) = \mathsf V(S_{58}, S_{24})$. On the other hand, it is easy
to see that both $L_2$ and $N_2$ can be embedded into $S_{24}$. Also, $S_{24}$ satisfies an equational basis of $\mathsf V(L_2, N_2)$ that can be found in \cite{sr}. It follows that $\mathsf V(S_{24})=\mathsf V(L_2, N_2)$, and so $\mathsf V(S_{(4, 475)}) = \mathsf V(S_{58}, L_2, N_2)$. Since $L_2$ can be embedded into $S_{58}$, we therefore have
\[
\mathsf V(S_{(4, 475)}) = \mathsf V(S_{58}, N_2).
\]
\end{remark}

\begin{cor}
The ai-semirings $S_{(4, 419)}$, $S_{(4, 440)}$ and $S_{(4, 477)}$ are all finitely based.
\end{cor}
\begin{proof}
It is easy to see that $S_{(4, 477)}$ and $S_{(4, 475)}$  have dual multiplications.
By Proposition $\ref{pro47501}$ we immediately deduce that $S_{(4, 477)}$ is finitely based.
On the other hand,
it is easy to verify that $S_{(4, 440)}$ satisfies the identity $(\ref{47501})$
and that $S_{(4, 475)}$ can be embedded into the direct product of two copies of $S_{(4, 440)}$.
Thus $\mathsf{V}(S_{(4, 440)})=\mathsf{V}(S_{(4, 475)})$ and so by Proposition \ref{pro47501}
$S_{(4, 440)}$ is finitely based.
Since $S_{(4, 419)}$ and $S_{(4, 440)}$ have dual multiplications, it follows immediately that $S_{(4, 419)}$ is also finitely based.
\end{proof}

\section{Equational basis of some $4$-element ai-semirings that relate to $S_{59}$}
In this section we shall show that some $4$-element ai-semirings that relate to $S_{59}$ are finitely based.
We first provide some information about the identities of $S_{59}$.

\begin{lem}\label{lem5901}
Let $\bu\approx \bu+\bq$ be a nontrivial ai-semiring identity such that
$\bu=\bu_1+\bu_2+\cdots+\bu_n$ and $\bu_i, \bq \in X^+$, $1\leq i \leq n$.
If $\bu\approx \bu+\bq$ is satisfied by $S_{59}$, then $\bu$ and $\bq$ satisfy one of the following conditions:
\begin{itemize}
\item[$(1)$] $\ell(\bu_i)\geq 3$ for some $\bu_i \in \bu$;

\item[$(2)$] $\ell(\bu_i)\leq 2$ for all $\bu_i \in \bu$. Then $\ell(\bq)\leq 2$.
If $\ell(\bq)=1$, then $c(\bq)\subseteq c(\bu)$. If $\ell(\bq)=2$, then $c(\bq)\subseteq c(L_2(\bu))$.
\end{itemize}
\end{lem}
\begin{proof}
Assume that $S_{59}$ satisfies $\bu\approx \bu+\bq$, where $\ell(\bu_i)\leq 2$ for all $\bu_i \in \bu$.
Since $T_2$ can be embedded into $S_{59}$,
it follows that $T_2$ satisfies $\bu\approx \bu+\bq$ and so $L_2(\bu)\neq\emptyset$.
We first show that $\ell(\bq)\leq 2$.
Let $\varphi: P_f(X^+) \to S_{59}$ be a substitution such that
$\varphi(x)=2$ for every $x\in c(L_2(\bu))$, $\varphi(x)=1$ for every $x\in c(L_1(\bu))\setminus c(L_2(\bu))$,
and $\varphi(x)=3$ otherwise.
It is easy to see that $\varphi(\bu)=1$ and so $\varphi(\bq)=1$.
Suppose by way of contradiction that $\ell(\bq)\geq 3$.
Then $\varphi(\bq)=3$,
since the identity $x_1x_2x_3\approx y_1y_2y_3$ and $3$ is the multiplicative zero element of $S_{59}$.
This is a contradiction. So $\ell(\bq)\leq 2$.
If $c(\bq)\nsubseteq c(\bu)$, then $\varphi(x)=3$ for some $x\in c(\bq)$
and so $\varphi(\bq)=3$.
This also contradicts with the fact that $\varphi(\bq)=1$. Thus $c(\bq)\subseteq c(\bu)$.
Now let $\ell(\bq)=2$. If $c(\bq)\nsubseteq c(L_2(\bu))$, then
$\varphi(x)=1$ for some $x\in c(\bq)$ and so $\varphi(\bq)=3$.
Therefore, $c(\bq)\subseteq c(L_2(\bu))$.
\end{proof}

\begin{pro}\label{pro41101}
$\mathsf{V}(S_{(4, 411)})$ is the ai-semiring variety defined by the identities
\begin{align}
xy &\approx xy+x; \label{41102}\\
x+yx &\approx x+yx+xy; \label{41100}\\
x_1x_2x_3 & \approx x_1x_2x_3+x_4; \label{41104}\\
x_1x_2+x_3x_4 &\approx x_1x_2+x_3x_4+x_1x_3; \label{41105}\\
x_1x_2+x_3x_4 &\approx x_1x_2+x_3x_4+x_1x_4. \label{41106}
\end{align}
\end{pro}
\begin{proof}
It is easy to verify that $S_{(4, 411)}$ satisfies the identities (\ref{41102})--(\ref{41106}).
In the remainder it is enough to show that every ai-semiring identity of $S_{(4, 411)}$
is derivable from (\ref{41102})--(\ref{41106}) and the identities defining $\mathbf{AI}$.
Let $\bu \approx \bu+\bq$ be such a nontrivial identity,
where $\bu=\bu_1+\bu_2+\cdots+\bu_n$ and $\bu_i, \bq\in X^+$, $1 \leq i \leq n$.
Since $T_2$ is isomorphic to $\{1, 4\}$,
it follows that $T_2$ satisfies $\bu \approx \bu+\bq$ and so $\ell(\bu_j)\geq 2$ for some $\bu_j \in \bu$.
It is easy to see that $S_{59}$ is isomorphic to $\{1, 2, 3\}$ and so $S_{59}$ satisfies $\bu \approx \bu+\bq$.
By Lemma \ref{lem5901} we need to consider the following two cases.

\textbf{Case 1.} $\ell(\bu_i) \geq 3$ for some $\bu_i \in \bu$. Then
\[
\bu \approx \bu+\bu_i \stackrel{(\ref{41104})}\approx
\bu+\bu_i+\bq \approx \bu+\bq.
\]

\textbf{Case 2.} $\ell(\bu_i)\leq 2$ for all $\bu_i \in \bu$.
Then $\ell(\bq)\leq 2$.
We shall show that $h(\bu_k)=h(\bq)$ for some $\bu_k \in \bu$.
If this is not true,
let us consider the semiring homomorphism
$\varphi: P_f(X^+) \to S_{(4, 411)}$ defined by
$\varphi(x)=4$ if $x = h(\bq)$ and $\varphi(x)=2$ otherwise.
It is easy to see that $\varphi(\bu)=3$ and $\varphi(\bq)=1$ or $4$.
So $\varphi(\bu) \neq \varphi(\bu+\bq)$,
a contradiction.
Thus there exists $\bu_k \in \bu$ such that $h(\bu_k)=h(\bq)$ and so
$\bu_k=h(\bq)s(\bu_k)$.

\textbf{Subcase 2.1.} $\ell(\bq)=1$. Then $\bu_k=\bq s(\bu_k)$.
Since $\bu \approx \bu+\bq$ is nontrivial, it follows that $\ell(\bu_k)=2$ and so $s(\bu_k)$ is nonempty.
Now we have
\[
\bu \approx \bu+\bu_k\approx \bu+\bq s(\bu_k)\stackrel{(\ref{41102})}\approx \bu+\bq s(\bu_k)+\bq.
\]

\textbf{Subcase 2.2.} $\ell(\bq)=2$.
Then by Lemma \ref{lem5901} $c(\bq)\subseteq c(L_2(\bu))$.
We may write $\bq=xy$.
Then $xx_1\in L_2(\bu)$ or $x_2x\in L_2(\bu)$ for some $x_1, x_2 \in X$,
$yy_1\in L_2(\bu)$ or $y_2y\in L_2(\bu)$ for some $y_1, y_2 \in X$.
If $\ell(\bu_k)=2$, then $s(\bu_k)$ is nonempty and so
\[
\bu \approx \bu+\bu_k+yy_1\approx \bu+xs(\bu_k)+yy_1\stackrel{(\ref{41105})}\approx \bu+xs(\bu_k)+yy_1+xy \approx \bu+\bq
\]
or
\[
\bu \approx \bu+\bu_k+y_1y\approx \bu+xs(\bu_k)+y_1y\stackrel{(\ref{41106})}\approx \bu+xs(\bu_k)+y_1y+xy \approx \bu+\bq.
\]
Now let $\ell(\bu_k)=1$. Then $\bu_k=x$.
If $xx_1\in L_2(\bu)$, then the remaining steps are similar to the preceding case.
If $x_2x\in L_2(\bu)$, then
\[
\bu\approx \bu+x+x_2x\stackrel{(\ref{41100})}\approx \bu+x+x_2x+xx_2.
\]
The remaining steps are similar to the preceding case.
\end{proof}

\begin{cor}
The ai-semiring $S_{(4, 412)}$ is finitely based.
\end{cor}
\begin{proof}
It is easy to see that $S_{(4, 412)}$ and $S_{(4, 411)}$  have dual multiplications.
By Proposition $\ref{pro41101}$ we immediately deduce that $S_{(4, 412)}$ is finitely based.
\end{proof}

\begin{pro}\label{pro41301}
$\mathsf{V}(S_{(4, 413)})$ is the ai-semiring variety defined by the identities
\begin{align}
xy &\approx yx; \label{41301}\\
x^2&\approx x^2+x; \label{41302}\\
x+xy&\approx x^2+xy; \label{41300}\\
x_1x_2x_3 & \approx x_1x_2x_3+x_4; \label{41303}\\
x^2+yz &\approx x^2+yz+xy. \label{41304}
\end{align}
\end{pro}
\begin{proof}
It is easy to check that $S_{(4, 413)}$ satisfies the identities (\ref{41301})--(\ref{41304}).
In the remainder we need only prove that every ai-semiring identity of $S_{(4, 413)}$
is derivable from (\ref{41301})--(\ref{41304}) and the identities defining $\mathbf{AI}$.
Let $\bu \approx \bu+\bq$ be such a nontrivial identity,
where $\bu=\bu_1+\bu_2+\cdots+\bu_n$. Since the identity (\ref{41301}) is satisfied by $S_{(4, 413)}$,
we may assume that $\bu_i, \bq\in X_c^+$, $1 \leq i \leq n$.
It is easy to see that $T_2$ is isomorphic to $\{1, 4\}$ and so $T_2$ satisfies $\bu \approx \bu+\bq$.
This implies that $\ell(\bu_j)\geq 2$ for some $\bu_j \in \bu$.
Since $S_{59}$ is isomorphic to $\{1, 2, 3\}$, it follows that $S_{59}$ also satisfies $\bu \approx \bu+\bq$.
By Lemma \ref{lem5901} we consider the following two cases.

\textbf{Case 1.} $\ell(\bu_i) \geq 3$ for some $\bu_i \in \bu$. Then
\[
\bu \approx \bu+\bu_i \stackrel{(\ref{41303})}\approx
\bu+\bu_i+\bq \approx \bu+\bq.
\]

\textbf{Case 2.} $\ell(\bu_i)\leq 2$ for all $\bu_i \in \bu$. Then $\ell(\bq)\leq 2$.
Suppose by way of contradiction that $c(\bu_k)\not\subseteq c(\bq)$ for all $\bu_k \in \bu$.
Consider the semiring homomorphism $\varphi: P_f(X^+) \to S_{(4, 413)}$ defined by
$\varphi(x)=4$ for every $x\in c(\bq)$ and $\varphi(x)=2$ otherwise.
It is easy to see that $\varphi(\bu)=3$ and $\varphi(\bq)=4$ or $1$.
This implies that $\varphi(\bu) \neq \varphi(\bu+\bq)$, a contradiction.
Thus $c(\bu_k) \subseteq c(\bq)$ for some $\bu_k \in \bu$.

\textbf{Subcase 2.1.} $\ell(\bq)=1$.
Since $\bu \approx \bu+\bq$ is nontrivial, it follows that $\bu_k=\bq^2$ and so
\[
\bu \approx \bu+\bu_k\approx \bu+\bq^2\stackrel{(\ref{41302})}\approx \bu+\bq^2+\bq.
\]
This derives the identity $\bu \approx \bu+\bq$.

\textbf{Subcase 2.2.} $\ell(\bq)=2$. Then by Lemma \ref{lem5901} $c(\bq)\subseteq c(L_2(\bu))$.
Let $\bq=xy$. Then $xx_1, yy_1\in L_2(\bu)$ for some $x_1, y_1 \in X$.
If $\ell(\bu_k)=2$, then $\bu_k=x^2$ or $y^2$.
We only prove the case that $\bu_k=x^2$. Now we have
\[
\bu \approx \bu+x^2+yy_1\stackrel{(\ref{41304})}\approx \bu+x^2+yy_1+xy\approx \bu+x^2+yy_1+\bq.
\]
If $\ell(\bu_k)=1$, then $\bu_k=x$ or $y$.
We only prove the case that $\bu_k=x$.
Indeed, we have
\[
\bu \approx \bu+\bu_k+xx_1 \approx \bu+x+xx_1\stackrel{(\ref{41300})}\approx \bu+x^2+xx_1.
\]
The remaining steps are similar to the preceding case.
\end{proof}

\begin{pro}\label{pro41401}
$\mathsf{V}(S_{(4, 414)})$ is the ai-semiring variety defined by the identities
\begin{align}
xy &\approx yx; \label{41401}\\
x_1x_2x_3 & \approx x_1x_2x_3+x_4; \label{41403}\\
x_1x_2+x_3x_4 &\approx x_1x_2+x_3x_4+x_1x_4. \label{41404}
\end{align}
\end{pro}
\begin{proof}
It is easily verified that $S_{(4, 414)}$ satisfies the identities (\ref{41401})--(\ref{41404}).
In the remainder we need only prove that every ai-semiring identity of $S_{(4, 414)}$
can be derived by (\ref{41401})--(\ref{41404}) and the identities defining $\mathbf{AI}$.
Let $\bu \approx \bu+\bq$ be such a nontrivial identity, where $\bu=\bu_1+\bu_2+\cdots+\bu_n$.
Since the identity (\ref{41401}) is satisfied by $S_{(4, 414)}$,
we may assume that $\bu_i, \bq\in X_c^+$, $1 \leq i \leq n$.
Since $T_2$ is isomorphic to $\{1, 3\}$, it follows that $T_2$ satisfies $\bu \approx \bu+\bq$ and so $\ell(\bu_j)\geq 2$ for some $\bu_j \in \bu$.
Since $S_{59}$ is isomorphic to $\{1, 2, 3\}$, we have that $S_{59}$ also satisfies $\bu \approx \bu+\bq$.
By Lemma \ref{lem5901} it is enough to consider the following two cases.

\textbf{Case 1.} $\ell(\bu_i) \geq 3$ for some $\bu_i \in \bu$. Then
\[
\bu \approx \bu+\bu_i \stackrel{(\ref{41403})}\approx \bu+\bu_i+\bq \approx \bu+\bq.
\]

\textbf{Case 2.} $\ell(\bu_i)\leq 2$ for all $\bu_i \in \bu$. Then $\ell(\bq)\leq 2$.
We shall show that $\ell(\bq)=2$.
Indeed, if $\ell(\bq)=1$, then $c(\bq)\subseteq c(\bu)$.
Let $\varphi: P_f(X^+) \to S_{(4, 414)}$ be a semiring homomorphism defined by
$\varphi(x)=4$ for every $x\in c(\bq)$ and $\varphi(x)=2$ otherwise.
It is easy to see that $\varphi(\bq)=4$ and $\varphi(L_2(\bu))=3$.
This implies that $\varphi(\bu+\bq)=1$ and so $\varphi(\bu)=1$.
Thus $\varphi(L_1(\bu))=4$ and so $c(\bq)\subseteq c(L_1(\bu))$.
This shows that $\bu \approx \bu+\bq$ is trivial, a contradiction.
So we have proved that $\ell(\bq)=2$. By Lemma \ref{lem5901} it follows that $c(\bq)\subseteq c(L_2(\bu))$.
We may write that $\bq=xy$. Then $xx_1$ and $yy_1\in L_2(\bu)$ for some $x_1, y_1\in X$.
Furthermore, we have
\[
\bu \approx \bu+xx_1+yy_1\stackrel{(\ref{41401}), (\ref{41404})}\approx \bu+xx_1+yy_1+xy\approx \bu+xx_1+yy_1+\bq.
\]
This proves the identity $\bu \approx \bu+\bq$.
\end{proof}

\section{Equational basis of some $4$-element ai-semirings that relate to $S_{60}$}
In this section we show that some $4$-element ai-semirings that relate to $S_{60}$ are finitely based.
This requires that we be able to provide some information of identities of $S_{60}$.

\begin{lem}\label{lem6001}
Let $\bu\approx \bu+\bq$ be a nontrivial ai-semiring identity such that
$\bu=\bu_1+\bu_2+\cdots+\bu_n$
and $\bu_i, \bq \in X^+$, $1\leq i \leq n$.
Suppose that $\bu\approx \bu+\bq$ is satisfied by $S_{60}$.
Then $\bu$ and $\bq$ satisfy the following conditions:
\begin{itemize}
\item[$(1)$] $\ell(\bu_i)\geq 2$ for some $\bu_i \in \bu$;

\item[$(2)$] If $\ell(\bq)=1$, then $c(\bq)\subseteq c(\bu)$;

\item[$(3)$] If $\ell(\bq)\geq 2$, then $c(\bq)\subseteq c(L_{\geq2}(\bu))$.
\end{itemize}
\end{lem}
\begin{proof}
Suppose that $S_{60}$ satisfies $\bu\approx \bu+\bq$.
It is easy to check that $T_2$ is isomorphic to $\{1, 3\}$
and so $T_2$ satisfies $\bu \approx \bu+\bq$.
This implies that $\ell(\bu_i)\geq 2$ for some $\bu_i \in \bu$.
Since $M_2$ is isomorphic to $\{2, 3\}$, it follows that
$M_2$ satisfies  $\bu \approx \bu+\bq$ and so $c(\bq)\subseteq c(\bu)$.
Let $\ell(\bq)\geq 2$. If $c(\bq) \nsubseteq c(L_{\geq2}(\bu))$,
then there exists $x\in c(\bq)$ such that $x\notin c(L_{\geq2}(\bu))$, but $x\in c(L_{\geq1}(\bu))$.
Consider the homomorphism $\varphi: P_f(X^+) \to S_{60}$ defined by
$\varphi(y)=1$ if $y=x$ and $\varphi(y)=2$ otherwise.
It is easy to see that $\varphi(\bu)=1$ and $\varphi(\bq)=3$.
So $\varphi(\bu)\neq \varphi(\bu+\bq)$, a contradiction.
Thus $c(\bq)\subseteq c(L_{\geq2}(\bu))$.
\end{proof}

\begin{pro}\label{pro45901}
$\mathsf{V}(S_{(4, 459)})$ is the ai-semiring variety defined by the identities
\begin{align}
x^2y&\approx xy; \label{45901}\\
xyz &\approx xzy; \label{45902}\\
xy &\approx xy+x; \label{45904}\\
x+yxz & \approx yxz+xzy; \label{45905}\\
x_1x_2+x_3x_4 & \approx x_1x_2+x_3x_4+x_1x_2x_3x_4, \label{45906}
\end{align}
where $y$ and $z$ may be empty in $(\ref{45905})$.
\end{pro}
\begin{proof}
It is easy to check that $S_{(4, 459)}$ satisfies the identities (\ref{45901})--(\ref{45906}).
In the remainder we need only prove that every ai-semiring identity of $S_{(4, 459)}$
can be derived by $(\ref{45901})-(\ref{45906})$ and the identities defining $\mathbf{AI}$.
Let $\bu \approx \bu+\bq$ be such a nontrivial identity,
where $\bu=\bu_1+\bu_2+\cdots+\bu_n$ and $\bu_i, \bq\in X^+$, $1 \leq i \leq n$.
It is easy to see that $L_2$ is isomorphic to $\{1, 3\}$ and so $L_2$ satisfies  $\bu \approx \bu+\bq$.
This implies that there exists $\bu_i \in \bu$ such that $h(\bu_i)=h(\bq)$ and so $\bu_i=h(\bq)s(\bu_i)$.
On the other hand, we have that there exists a congruence $\rho=\{(1,3),(3,1)\}\cup\vartriangle$
 such that $S_{(4, 459)}/\rho$ is isomorphic to $S_{60}$ and so $S_{60}$ satisfies $\bu \approx \bu+\bq$.
By Lemma \ref{lem6001} we only need to consider the following two cases.

\textbf{Case 1.} $\ell(\bq)=1$. Then $\ell(\bu_i)\geq 2$ and $\bu_i=h(\bq)s(\bu_i)=\bq s(\bu_i)$,
where $s(\bu_i)$ is nonempty.
Furthermore, we have
\[
\bu \approx \bu+\bu_i \approx
\bu+h(\bq)s(\bu_i)\approx \bu+\bq s(\bu_i)
\stackrel{(\ref{45904})}\approx \bu+\bq s(\bu_i)+\bq.
\]
This derives the identity $\bu \approx \bu+\bq$.

\textbf{Case 2.} $\ell(\bq)\geq 2$. By Lemma ${\ref{lem6001}}$ it follows that $c(\bq)\subseteq c(L_{\geq2}(\bu))$.
Consider the following two subcases.

\textbf{Subcase 2.1.} $\ell(\bu_i)\geq 2$. Then $\bu_i\in L_{\geq2}(\bu)$.
Let $L_{\geq2}(\bu)=\{\bu_j \mid 1\leq j\leq m\}$.
Now we have
\begin{align*}
\bu
&\approx \bu+\bu_i+\bu_1+\cdots+\bu_m\\
&\approx \bu+\bu_i+\bu_1+\cdots+\bu_m+\bu_i\bu_1\cdots\bu_m. &&(\text{by}~(\ref{45906}))
\end{align*}
This implies the identity
\[
\bu\approx \bu+\bu_i\bu_1\cdots\bu_m.
\]
Furthermore, we can deduce
\begin{align*}
\bu
&\approx \bu+\bu_i\bu_1\cdots\bu_m\\
&\approx \bu+h(\bq) s(\bu_i)\bu_1\cdots\bu_m  \\
&\approx  \bu+\bq\bp &&(\text{by}~(\ref{45901}), (\ref{45902}))\\
&\approx  \bu+\bq\bp+\bq. &&(\text{by}~(\ref{45904}))
\end{align*}

\textbf{Subcase 2.2.} $\ell(\bu_i)=1$. Then $\bu_i=h(\bq)$ and so $c(\bu_i)\subseteq c(L_{\geq2}(\bu))$.
It follows that there exists $\bu_j \in \bu$ such that
$\bu_j=\bp_1h(\bq)\bp_2$, where $\bp_1$ and $\bp_2$ can not be empty simultaneously. Now we have
\begin{align*}
\bu
&\approx \bu+\bu_i+\bu_j\\
&\approx \bu+h(\bq)+\bp_1h(\bq)\bp_2  \\
&\approx  \bu+\bp_1h(\bq)\bp_2+h(\bq)\bp_2\bp_1. &&(\text{by}~(\ref{45905}))
\end{align*}
This implies the identity
\[
\bu\approx \bu+h(\bq)\bp_2\bp_1,
\]
where $\ell(h(\bq)\bp_2\bp_1)\geq 2$.
The remaining steps are similar to the preceding case.
\end{proof}

\begin{remark}
It is a routine matter to verify that $S_{(4, 459)}$ is isomorphic to a subdirect product of $S_{60}$ and $S_{23}$. So $\mathsf V(S_{(4, 459)}) = \mathsf V(S_{60}, S_{23})$. On the other hand, it is easy
to see that both $L_2$ and $M_2$ can be embedded into $S_{23}$.
Also, $S_{23}$ satisfies an equational basis of $\mathsf V(L_2, M_2)$ that can be found in \cite{sr}.
It follows that $\mathsf V(S_{23})=\mathsf V(L_2, M_2)$, and so $\mathsf V(S_{(4, 459)})= \mathsf V(S_{60}, L_2, M_2)$. Since $M_2$ can be embedded into $S_{60}$, we therefore have
\[
\mathsf{V}(S_{(4, 459)})=\mathsf{V}(S_{60}, L_2).
\]
\end{remark}

\begin{cor}
The ai-semirings $S_{(4, 394)}$, $S_{(4, 397)}$ and $S_{(4, 423)}$ are all finitely based.
\end{cor}
\begin{proof}
Since $S_{(4, 423)}$ and $S_{(4, 459)}$ have dual multiplications, it follows from
Proposition {\ref{pro45901}} that $S_{(4, 423)}$ is finitely based.
It is easy to verify that $S_{(4, 397)}$ satisfies the identities (\ref{45901})--(\ref{45906})
and that $S_{(4, 459)}$ can be embedded into the direct product of two copies of $S_{(4, 397)}$.
So $\mathsf{V}(S_{(4, 459)}) = \mathsf{V}(S_{(4, 397)})$.
By Proposition {\ref{pro45901}} it follows that $S_{(4, 397)}$ is finitely based.
Since $S_{(4, 394)}$ and $S_{(4, 397)}$ have dual multiplications,
we immediately deduce that $S_{(4, 394)}$ is also finitely based.
\end{proof}

\begin{pro}\label{pro46701}
$\mathsf{V}(S_{(4, 467)})$ is the ai-semiring variety defined by the identities
\begin{align}
xy &\approx yx; \label{46701}\\
x^2y &\approx xy; \label{46702}\\
x^2 &\approx x^2+x; \label{46703}\\
x+xyz & \approx x+xyz+xy;\label{46704}\\
x_1x_2+x_3x_4 & \approx x_1x_2+x_3x_4+x_1x_2x_3x_4. \label{46705}
\end{align}
\end{pro}
\begin{proof}
It is easily verified that $S_{(4, 467)}$ satisfies the identities (\ref{46701})--(\ref{46705}).
In the remainder we need only prove that every ai-semiring identity of $S_{(4, 467)}$
can be derived by (\ref{46701})--(\ref{46705}) and the identities defining $\mathbf{AI}$.
Let $\bu \approx \bu+\bq$ be such a nontrivial identity,
where $\bu=\bu_1+\bu_2+\cdots+\bu_n$ and $\bu_i, \bq\in X^+$, $1 \leq i \leq n$.
It is easy to see that $D_2$ is isomorphic to $\{1, 3\}$ and so $D_2$ satisfies  $\bu \approx \bu+\bq$.
This implies that there exists $\bu_r \in \bu$ such that $c(\bu_r)\subseteq c(\bq)$. 
On the other hand, we have that there exists a congruence $\rho=\{(1,3),(3,1)\}\cup \vartriangle$
 such that $S_{(4, 467)}/\rho$ is isomorphic to $S_{60}$ and so $S_{60}$ satisfies $\bu \approx \bu+\bq$.
By Lemma \ref{lem6001} it is enough to consider the following two cases.

\textbf{Case 1.} $\ell(\bq)=1$.
Since $\bu \approx \bu+\bq$ is nontrivial,
it follows that $\ell(\bu_r)\geq 2$ and so
$\bu_r=\bq^k$ for some $k\geq 2$. Now we have
\[
\bu \approx \bu+\bu_r \approx
\bu+\bq^k\stackrel{(\ref{46703})}\approx \bu+\bq^k+\bq.
\]
This implies the identity $\bu \approx \bu+\bq$.

\textbf{Case 2.} $\ell(\bq)\geq 2$.
Then $c(\bq)\subseteq c(L_{\geq2}(\bu))$ and so $c(\bu_r)\subseteq c(L_{\geq2}(\bu))$.
Assume that $L_{\geq2}(\bu)=\{\bu_1, \bu_2, \ldots, \bu_m\}$.
We have
\begin{align*}
\bu
&\approx \bu+\bu_1+\cdots+\bu_m\\
&\approx \bu+\bu_1+\cdots+\bu_m+\bu_1\cdots\bu_m&&(\text{by}~(\ref{46705}))\\
&\approx \bu+\bu_1+\cdots+\bu_m+\bu_r\bu_1\cdots\bu_m.  &&(\text{by}~(\ref{46701}), (\ref{46702}))
\end{align*}
This derives the identity
\begin{equation}\label{id24082025}
\bu\approx \bu+\bu_r\bu_1\cdots\bu_m.
\end{equation}
Furthermore, we can deduce
\begin{align*}
\bu
&\approx \bu+\bu_r+\bu_r\bu_1\cdots\bu_m &&(\text{by}~(\ref{id24082025}))\\
&\approx \bu+\bu_r+\bu_r\bq\bp &&(\text{by}~(\ref{46701}), (\ref{46702})) \\
&\approx  \bu+\bu_r+\bu_r\bq\bp+\bu_r\bq &&(\text{by}~(\ref{46704}))\\
&\approx \bu+\bu_r+\bu_r\bq\bp+\bq.&&(\text{by}~(\ref{46701}), (\ref{46702}))
\end{align*}
This proves the identity $\bu\approx \bu+\bq$.
\end{proof}
\begin{remark}
It is a routine matter to verify that $S_{(4, 467)}$ is isomorphic to a subdirect product of $S_{60}$ and $S_{22}$.
So $\mathsf V(S_{(4, 467)}) = \mathsf V(S_{60}, S_{22})$.
On the other hand, it is easy
to see that both $M_2$ and $D_2$ can be embedded into $S_{22}$.
Also, $S_{22}$ satisfies an equational basis of $\mathsf V(D_2, M_2)$ that can be found in \cite{sr}.
It follows that $\mathsf V(S_{22})=\mathsf V(D_2, M_2)$, and so
$\mathsf V(S_{(4, 467)}) = \mathsf V(S_{60}, D_2, M_2)$. Since $M_2$ can be embedded into $S_{60}$, we therefore obtain
\[
\mathsf{V}(S_{(4, 467)})=\mathsf{V}(S_{60}, D_2).
\]
\end{remark}

\begin{cor}
The ai-semiring $S_{(4, 393)}$ is finitely based.
\end{cor}
\begin{proof}
It is easy to check that $S_{(4, 393)}$ satisfies the identities (\ref{46701})--(\ref{46705})
and that $S_{(4, 467)}$ can be embedded into the direct product of two copies of $S_{(4, 393)}$.
Thus $\mathsf{V}(S_{(4, 467)})=\mathsf{V}(S_{(4, 393)})$
and so by Proposition $\ref{pro46701}$ $S_{(4, 393)}$ is finitely based.
\end{proof}

\begin{pro}\label{pro47901}
$\mathsf{V}(S_{(4, 479)})$ is the ai-semiring variety defined by the identities
\begin{align}
x^3&\approx x^2; \label{47901}\\
xy &\approx yx; \label{47902}\\
xy&\approx x^2y; \label{id240818100}\\
xyz&\approx xyz+xy; \label{id24081915}\\
x_1x_2+x_3x_4& \approx x_1x_2x_3x_4. \label{id24081901}
\end{align}
\end{pro}
\begin{proof}
It is easy to verify that $S_{(4, 479)}$ satisfies the identities (\ref{47901})--(\ref{id24081901}).
In the remainder we need only show that every ai-semiring identity of $S_{(4, 479)}$
can be derived by (\ref{47901})--(\ref{id24081901}) and the identities defining $\mathbf{AI}$.
Let $\bu \approx \bu+\bq$ be such a nontrivial identity, where $\bu=\bu_1+\bu_2+\cdots+\bu_n$
and $\bu_i, \bq \in X^+$, $1 \leq i \leq n$.
It is easy to see that $N_2$ is isomorphic to $\{1, 3\}$ and so $N_2$ satisfies  $\bu \approx \bu+\bq$.
This implies that $\ell(\bq)\geq 2$.
On the other hand, we have that there exists a congruence $\rho=\{(1,3),(3,1)\}\cup \vartriangle$
 such that $S_{(4, 479)}/\rho$ is isomorphic to $S_{60}$ and so $S_{60}$ satisfies $\bu \approx \bu+\bq$.
By Lemma \ref{lem6001} we have that $c(\bq)\subseteq c(L_{\geq2}(\bu))$.
Assume that $L_{\geq2}(\bu)=\{\bu_1, \bu_2, \ldots, \bu_k\}$
and that $c(L_{\geq2}(\bu))=\{x_1, \ldots, x_m\}$.
By (\ref{47901}) and (\ref{47902}) we can deduce the identity
\begin{equation}\label{id24081905}
\bq \approx x_1^{r_1}x_2^{r_2}\cdots x_m^{r_m}
\end{equation}
for some $r_1, r_2, \ldots, r_m$,
where $0\leq r_i\leq 2$, $1\leq i \leq m$.
Now we have
\begin{align*}
\bu
&\approx \bu+\bu_1\bu_2\cdots\bu_k &&(\text{by}~(\ref{id24081901}))\\
&\approx \bu+x_1^2x_2^2\cdots x_m^2 &&(\text{by}~(\ref{47901}), (\ref{47902}), (\ref{id240818100}))\\
&\approx \bu+x_1^{r_1}x_2^{r_2}\cdots x_m^{r_m}\bp &&(\text{by}~(\ref{47902}))\\
&\approx \bu+x_1^{r_1}x_2^{r_2}\cdots x_m^{r_m}\bp+x_1^{r_1}x_2^{r_2}\cdots x_m^{r_m}&&(\text{by}~(\ref{id24081915}))\\
&\approx \bu+x_1^{r_1}x_2^{r_2}\cdots x_m^{r_m}\bp+\bq, &&(\text{by}~(\ref{id24081905}))
\end{align*}
where $\bp=x_1^{2-r_1}x_2^{2-r_2}\cdots x_m^{2-r_m}$.
This implies the identity $\bu \approx \bu+\bq$.
\end{proof}
\begin{remark}
It is a routine matter to verify that $S_{(4, 479)}$ is isomorphic to a subdirect product of $S_{60}$ and $S_{21}$.
So $\mathsf V(S_{(4, 479)}) = \mathsf V(S_{60}, S_{21})$. On the other hand, it is easy
to see that both $M_2$ and $N_2$ can be embedded into $S_{21}$. Also, $S_{21}$ satisfies an equational basis of $\mathsf V(N_2, M_2)$ that can be found in \cite{sr}. It follows that $\mathsf V(S_{21})=\mathsf V(N_2, M_2)$, and so $\mathsf V(S_{(4, 479)}) = \mathsf V(S_{60}, N_2, M_2)$. Since $M_2$ can be embedded into $S_{60}$, we therefore obtain
\[
\mathsf V(S_{(4, 479)}) = \mathsf V(S_{60}, N_2).
\]
\end{remark}

\begin{cor}
The ai-semiring $S_{(4, 392)}$ is finitely based.
\end{cor}
\begin{proof}
It is easy to verify that $S_{(4, 392)}$ satisfies the identities (\ref{47901})--(\ref{id24081901})
and that $S_{(4, 479)}$ can be embedded into the direct product of two copies of $S_{(4, 392)}$.
Thus $\mathsf{V}(S_{(4, 479)})=\mathsf{V}(S_{(4, 392)})$ and so by
Proposition $\ref{pro47901}$ $S_{(4, 392)}$ is finitely based.
\end{proof}

\begin{pro}\label{pro39001}
$\mathsf{V}(S_{(4, 390)})$ is the ai-semiring variety defined by the identities
\begin{align}
x^2y& \approx xy; \label{39001}\\
xyz& \approx yxz; \label{39002}\\
x^2y^2& \approx {(xy)}^2; \label{39003}\\
x^2y^2& \approx y^2x^2; \label{39000}\\
xyz& \approx xyz+y; \label{39005}\\
x^2+yz& \approx x^2yz; \label{39007}\\
x_1x_2^2+x_3x_4^2& \approx x_1x_2x_3x_4^2; \label{39009}\\
x_1x_2+x_3x_4& \approx x_1x_2+x_3x_4+x_1x_3^2, \label{39010}
\end{align}
where $x$ and $z$ may be empty in $(\ref{39005})$.
\end{pro}
\begin{proof}
It is easy to check that $S_{(4, 390)}$ satisfies the identities (\ref{39001})--(\ref{39010}).
In the remainder it is enough to show that every ai-semiring identity of $S_{(4, 390)}$
is derivable from (\ref{39001})--(\ref{39010}) and the identities defining $\mathbf{AI}$.
Let $\bu \approx \bu+\bq$ be such a nontrivial identity,
where $\bu=\bu_1+\bu_2+\cdots+\bu_n$ and $\bu_i, \bq\in X^+$, $1 \leq i \leq n$.
It is a routine matter to verify that $S_{(4, 390)}$ is isomorphic to
a subdirect product of $S_{57}$ and $S_{60}$.
So both $S_{57}$ and $S_{60}$ satisfy $\bu \approx \bu+\bq$.
By Lemma \ref{lem5701}
we obtain that $\ell(\bu_i)\geq 2$ for some $\bu_i \in \bu$,
$c(p(\bq))\subseteq c(p(\bu))$ and $t(\bq)\in c(\bu)$.

\textbf{Case 1.} $\ell(\bq)=1$. Then $c(\bq)\subseteq c(\bu_k)$ for some $\bu_k \in \bu$
and so $\bu_k=\bp_1\bq\bp_2$ for some $\bp_1, \bp_2 \in X^*$.
Furthermore, we have
\[
\bu \approx \bu+\bu_k \approx
\bu+\bp_1\bq\bp_2 \stackrel{(\ref{39005})}\approx \bu+\bp_1\bq\bp_2+\bq.
\]
This implies the identity $\bu \approx \bu+\bq$.

\textbf{Case 2.} $\ell(\bq)\geq 2$.
By Lemma \ref{lem6001} it follows that $c(\bq)\subseteq c(L_{\geq2}(\bu))$.
Suppose that $L_{\geq2}(\bu)=\{\bu_1, \bu_2, \ldots, \bu_m\}$ and that
$c(p(L_{\geq2}(\bu)))=\{x_1, x_2, \ldots, x_s\}$. Then
\begin{align*}
\bu
&\approx \bu+\bu_1+\bu_2+\cdots+\bu_m \\
&\approx \bu+p(\bu_1)t(\bu_1)+p(\bu_2)t(\bu_2)+\cdots+p(\bu_m)t(\bu_m) \\
&\approx  \bu+p(\bu_1){p(\bu_m)}^2+p(\bu_2){p(\bu_m)}^2+\cdots+p(\bu_{m-1}){p(\bu_m)}^2 &&(\text{by}~(\ref{39010}))\\
&\approx  \bu+{p(\bu_1)}^2{p(\bu_2)}^2{p(\bu_3)}^2{p(\bu_4)}^2\cdots {p(\bu_{m-1})}^2{p(\bu_m)}^2
&&(\text{by}~(\ref{39001}), (\ref{39002}), (\ref{39009}))\\
&\approx \bu+x_1^2x_2^2\cdots x_s^2. &&(\text{by}~(\ref{39001}), (\ref{39003}), (\ref{39000}))
\end{align*}
This derives the identity
\begin{equation}\label{240818001}
\bu \approx \bu+x_1^2x_2^2\cdots x_s^2.
\end{equation}
If $t(\bq)\in c(p(\bu))$, then
\[
\bu \stackrel{(\ref{240818001})}\approx \bu+x_1^2x_2^2\cdots x_s^2
\stackrel{(\ref{39001}), (\ref{39002})}\approx \bu+\bq\bp \stackrel{(\ref{39005})}\approx \bu+\bq\bp+\bq.
\]
So we obtain the identity $\bu \approx \bu+\bq$.
If $t(\bq)\notin c(p(\bu))$, then
$m(t(\bq), \bq)=1$ and
$t(\bq)=t(\bu_\ell)$ for some $\bu_\ell \in L_{\geq2}(\bu)$.
Furthermore, we have
\begin{align*}
\bu
&\approx \bu+x_1^2x_2^2\cdots x_s^2+\bu_\ell&&(\text{by}~(\ref{240818001}))\\
&\approx \bu+x_1^2x_2^2\cdots x_s^2+p(\bu_\ell)t(\bu_\ell)\\
&\approx \bu+x_1^2x_2^2\cdots x_s^2+p(\bu_\ell)t(\bq)\\
&\approx \bu+x_1^2x_2^2\cdots x_s^2p(\bu_\ell)t(\bq) &&(\text{by}~(\ref{39007}))\\
&\approx \bu+\bp\bq &&(\text{by}~(\ref{39001}), (\ref{39002}))\\
&\approx \bu+\bp\bq+\bq. &&(\text{by}~(\ref{39005}))
\end{align*}
This derives the identity $\bu\approx \bu+\bq$.
\end{proof}

\begin{cor}
The ai-semiring $S_{(4, 396)}$ is finitely based.
\end{cor}
\begin{proof}
It is easy to see that $S_{(4, 396)}$ and $S_{(4, 390)}$ have dual multiplications.
By Proposition $\ref{pro39001}$ we immediately deduce that $S_{(4, 396)}$ is finitely based.
\end{proof}

\begin{pro}\label{pro39801}
$\mathsf{V}(S_{(4, 398)})$ is the ai-semiring variety defined by the identities
\begin{align}
xy& \approx yx; \label{39801}\\
xy& \approx xy+x; \label{39802}\\
x^2+yz& \approx x^2yz; \label{39803}\\
xyz& \approx xy+yz+xz. \label{39804}
\end{align}
\end{pro}
\begin{proof}
It is easy to verify that $S_{(4, 398)}$ satisfies the identities (\ref{39801})--(\ref{39804}).
In the remainder it is enough to show that every ai-semiring identity of $S_{(4, 398)}$
is derivable from (\ref{39801})--(\ref{39804}) and the identities defining $\mathbf{AI}$.
Let $\bu \approx \bu+\bq$ be such a nontrivial identity,
where $\bu=\bu_1+\bu_2+\cdots+\bu_n$ and $\bu_i, \bq\in X^+$, $1 \leq i \leq n$.
It is a routine matter to verify that $S_{(4, 398)}$ is isomorphic to
a subdirect product of $S_{53}$ and $S_{60}$.
So both $S_{53}$ and $S_{60}$ satisfy $\bu \approx \bu+\bq$.

\textbf{Case 1.} $\ell(\bq)=1$. By Lemma \ref{lem6001} we obtain that $c(\bq)\subseteq c(\bu)$ and
so $c(\bq)\subseteq c(\bu_j)$ for some $\bu_j \in \bu$.
This implies that $\ell(\bu_j)\geq 2$.
By the identity $(\ref{39801})$ we deduce $\bu_j\approx \bq\bp$ for some nonempty word $\bp$ and so
\[
\bu \approx \bu+\bu_j \approx
\bu+\bq\bp \stackrel{(\ref{39802})}\approx \bu+\bq\bp+\bq.
\]
This implies $\bu \approx \bu+\bq$.

\textbf{Case 2.} $\ell(\bq)\geq 2$. Let $\bq=x_1x_2\cdots x_n$.
Then (\ref{39804}) implies the identity
\[
\bq\approx \sum\limits_{1\leq i\textless j\leq n}x_ix_j.
\]
So we only consider the case that $\bq=xy$.
By Lemma \ref{lem5301} we need to consider the following subcases.

\textbf{Subcase 2.1.} $x=y$. Then $m(x, \bu_k)\geq 2$ for some $\bu_k\in \bu$. So we have
\[
\bu
\approx \bu+\bu_k \stackrel{(\ref{39801})}\approx \bu+x^2\bu_k'
\stackrel{(\ref{39802})}\approx  \bu+x^2\bu_k'+x^2
\approx \bu+x^2\bu_k'+xy.
\]

\textbf{Subcase 2.2.} $x \neq y$.
Then there exists $\bu_k\in \bu$ such that either
$x, y\in c(\bu_k)$ or $m(x, \bu_k)\geq 2$ or $m(y, \bu_k)\geq 2$.
If $x, y \in c(\bu_k)$ for some $\bu_k\in \bu$, then
\[
\bu
\approx \bu+\bu_k\stackrel{(\ref{39801})}\approx\bu+xy\bu_k' \stackrel{(\ref{39802})}\approx \bu+xy\bu_k'+xy.
\]
If $m(x, \bu_k)\geq 2$, then by the similar argument in Subcase 2.1 one can deduce
\[
\bu\approx \bu+x^2.
\]
Since $y\in c(\bq)$, it follows from Lemma \ref{lem6001} that
$y\in c(\bu_\ell)$ for some $\bu_\ell\in L_{\geq 2}(\bu)$.
Now we have
\begin{align*}
\bu
&\approx \bu+x^2+\bu_\ell\\
&\approx \bu+x^2+y\bp&&(\text{by}~(\ref{39801}))\\
&\approx \bu+x^2y\bp &&(\text{by}~(\ref{39803}))\\
&\approx \bu+xyx\bp &&(\text{by}~(\ref{39801}))\\
&\approx \bu+xyx\bp+xy. &&(\text{by}~(\ref{39802}))
\end{align*}
If $m(y, \bu_k)\geq 2$, then the remaining steps are similar to the preceding case.
\end{proof}

\section{Equational basis of some $4$-element ai-semirings that relate to $S^0$}
In this section we provide equational basis for some $4$-element ai-semirings that relate to $S^0$.

\begin{pro}\label{pro47401}
$\mathsf{V}(S_{(4, 474)})$ is the ai-semiring variety defined by the identities
\begin{align}
xy & \approx yx; \label{47401}\\
x^2y & \approx xy; \label{47402}\\
xy &\approx xy+xyz. \label{47403}
\end{align}
\end{pro}
\begin{proof}
It is easy to verify that $S_{(4, 474)}$ satisfies the identities (\ref{47401})--(\ref{47403}).
In the remainder we need only prove that every ai-semiring identity of $S_{(4, 474)}$
can be derived by (\ref{47401})--(\ref{47403}) and the identities defining $\mathbf{AI}$.
Let $\bu \approx \bu+\bq$ be such a nontrivial identity, where $\bu=\bu_1+\bu_2+\cdots+\bu_n$
and $\bu_i, \bq\in X^+$, $1 \leq i \leq n$.
It is easy to see that $N_2$ is isomorphic to $\{1, 3\}$ and so $N_2$ satisfies  $\bu \approx \bu+\bq$.
This implies that $\ell(\bq)\geq 2$.
On the other hand, we have that there exists a congruence $\rho=\{(1,3),(3,1)\}\cup \vartriangle$
 such that $S_{(4, 474)}/\rho$ is isomorphic to $T_2^0$ and so $T_2^0$ satisfies $\bu \approx \bu+\bq$.
By Lemma \ref{lem001} there exists $\bu_i \in \bu$ such that
$\ell(\bu_i)\geq 2$ and $c(\bu_i) \subseteq c(\bq)$. So we have
\[
\bu \approx \bu+\bu_i \stackrel{(\ref{47403})}\approx \bu+\bu_i+\bu_i\bq \stackrel{(\ref{47401}), (\ref{47402})}\approx \bu+\bu_i+\bq.
\]
This derives the identity $\bu\approx\bu+\bq$.
\end{proof}

\begin{remark}
It is a routine matter to verify that $S_{(4, 474)}$ is isomorphic to a subdirect product of $T_2^0$ and $S_{34}$.
So $\mathsf{V}(S_{(4, 474)}) = \mathsf{V}(T_2^0, S_{34})$.
On the other hand, it is easy
to see that both $D_2$ and $N_2$ can be embedded into $S_{34}$.
Also, $S_{34}$ satisfies an equational basis of $\mathsf{V}(D_2, N_2)$ that can be found in \cite{sr}.
It follows that $\mathsf{V}(S_{34})=\mathsf{V}(D_2, N_2)$ and so $\mathsf{V}(S_{(4, 474)})=\mathsf{V}(T_2^0, D_2, N_2)$.
Since $D_2$ can be embedded into $T_2^0$, we therefore have
\[
\mathsf{V}(S_{(4, 474)})=\mathsf{V}(T_2^0, N_2).
\]
\end{remark}
\begin{cor}
The ai-semiring $S_{(4, 441)}$ is finitely based.
\end{cor}
\begin{proof}
It is easy to verify that $S_{(4, 441)}$ satisfies the identities (\ref{47401})--(\ref{47403})
and that $S_{(4, 474)}$ can be embedded into the direct product of two copies of $S_{(4, 441)}$.
So $\mathsf{V}(S_{(4, 441)})=\mathsf{V}(S_{(4, 474)})$.
By Proposition $\ref{pro47401}$ we immediately deduce that $S_{(4, 441)}$ is finitely based.
\end{proof}

\begin{pro}
$\mathsf{V}(S_{(4, 431)})$ is the ai-semiring variety defined by the identities
\begin{align}
x^2y & \approx xy; \label{43101}\\
xy & \approx yx; \label{43103}\\
x^2 &\approx x^2+x; \label{43104}\\
x+yz & \approx x+yz+xyz. \label{43105}
\end{align}
\end{pro}
\begin{proof}
It is easy to check that $S_{(4, 431)}$ satisfies the identities (\ref{43101})--(\ref{43105}).
In the remainder we need only prove that every ai-semiring identity of $S_{(4, 431)}$
can be derived by (\ref{43101})--(\ref{43105}) and the identities defining $\mathbf{AI}$.
Let $\bu \approx \bu+\bq$ be such a nontrivial identity,
where $\bu=\bu_1+\bu_2+\cdots+\bu_n$ and $\bu_i, \bq\in X^+$, $1 \leq i \leq n$.
It is a routine matter to verify that $S_{(4, 431)}$ is isomorphic to
a subdirect product of $M_2^0$ and $T_2^0$.
So both $M_2^0$ and $T_2^0$ satisfy $\bu \approx \bu+\bq$.
By Lemma \ref{lem001} we obtain that $c(\bq)= c(D_\bq(\bu))$ and that $L_{\geq 2}(\bu)\cap D_{\bq}(\bu)$ is nonempty.
Let $D_\bq(\bu)=\bu_1+\bu_2+\cdots+\bu_k$. Then $c(\bq)=c(\bu_1\bu_2\cdots\bu_k)$.

\textbf{Case 1.} $\ell(\bq)=1$. It follows that there exists $\bu_1\in \bu$ such that $\bu_1=\bq^k$ for some $k\geq 2$.
Now we have
\[
\bu \approx \bu+\bu_1 \approx
\bu+\bq^k\stackrel{(\ref{43104})}\approx \bu+\bq^k+\bq.
\]

\textbf{Case 2.} $\ell(\bq)\geq 2$. Since $L_{\geq 2}(\bu)\cap D_{\bq}(\bu)$ is nonempty, we deduce
\begin{align*}
\bu
&\approx \bu+\bu_1+\bu_2+\cdots+\bu_k \\
&\approx  \bu+\bu_1+\bu_2+\cdots+\bu_k+\bu_1\bu_2\cdots\bu_k &&(\text{by}~(\ref{43105}))\\
&\approx  \bu+\bu_1+\bu_2+\cdots+\bu_k+\bq. &&(\text{by}~(\ref{43101}), (\ref{43103}))
\end{align*}
This derives the identity $\bu \approx \bu+\bq$.
\end{proof}

\begin{pro}\label{pro44501}
$\mathsf{V}(S_{(4, 445)})$ is the ai-semiring variety defined by the identities
\begin{align}
x^2y & \approx xy; \label{44501}\\
xyz & \approx xzy; \label{44502}\\
x^2 &\approx x^2+x; \label{44503}\\
x_1+x_2x_3 & \approx x_1+x_2x_3+x_1x_2x_3x_4. \label{44505}
\end{align}
\end{pro}
\begin{proof}
It is easy to verify that $S_{(4, 445)}$ satisfies the identities (\ref{44501})--(\ref{44505}).
In the remainder it is enough to show that every ai-semiring identity of $S_{(4, 445)}$
can be derived by (\ref{44501})--(\ref{44505}) and the identities defining $\mathbf{AI}$.
Let $\bu \approx \bu+\bq$ be such a nontrivial identity,
where $\bu=\bu_1+\bu_2+\cdots+\bu_n$, where $\bu_i, \bq\in X^+$, $1 \leq i \leq n$.
It is a routine matter to verify that $S_{(4, 445)}$ is isomorphic to
a subdirect product of $L^0_2$ and $T^0_2$.
So both $L^0_2$ and $T^0_2$ satisfy $\bu \approx \bu+\bq$.
By Lemma \ref{lem001} there exist $\bu_i, \bu_j\in \bu$ such that
$c(\bu_i),  c(\bu_j)\subseteq c(\bq)$,
$h(\bu_i)=h(\bq)$ and $\ell(\bu_j)\geq 2$.

\textbf{Case 1.} $\ell(\bq)=1$. Then $\bu_i=\bq^k$ for some $k\geq 2$.
We can deduce
\[
\bu \approx \bu+\bu_i \approx
\bu+\bq^k\stackrel{(\ref{44503})}\approx \bu+\bq^k+\bq.
\]

\textbf{Case 2.} $\ell(\bq)\geq 2$. Then
\begin{align*}
\bu
&\approx \bu+\bu_i+\bu_j \\
&\approx  \bu+\bu_i+\bu_j+\bu_i\bu_j\bq^2&&(\text{by}~(\ref{44505}))\\
&\approx  \bu+\bu_i+\bu_j+h(\bq)s(\bu_i)\bu_j\bq^2\\
&\approx  \bu+\bu_i+\bu_j+\bq^2&&(\text{by}~(\ref{44501}), (\ref{44502}))\\
&\approx  \bu+\bu_i+\bu_j+\bq^2+\bq.&&(\text{by}~(\ref{44503}))
\end{align*}
This derives the identity $\bu \approx \bu+\bq$.
\end{proof}
\begin{cor}
The ai-semiring $S_{(4, 433)}$ is finitely based.
\begin{proof}
 It is easy to see that $S_{(4, 433)}$ and $S_{(4, 445)}$ have dual multiplications.
 By Proposition $\ref{pro44501}$ we immediately deduce that $S_{(4, 433)}$ is finitely based.
\end{proof}
\end{cor}

The following result, which is due to \cite[Lemma 4.1]{rlzc}, provides a solution of equational problem for $S_2$.
\begin{lem}\label{lem201}
Let $\bu\approx \bu+\bq$ be an ai-semiring identity such that
$\bu=\bu_1+\bu_2+\cdots+\bu_n$ and $\bu_i, \bq \in X^+$, $1\leq i \leq n$.
Then $\bu\approx \bu+\bq$ is satisfied by $S_2$ if and only if $\bu$ and $\bq$ satisfy one of the following conditions:
\begin{itemize}
\item[$(1)$] $\ell(\bu_i)\geq 3$ for some $\bu_i \in \bu$;

\item[$(2)$] $c(L_1(\bu))\cap c(L_2(\bu)) \neq\emptyset$;

\item[$(3)$] $\ell(\bu_i)\leq 2$ for all $\bu_i \in \bu$, $c(L_1(\bu))\cap c(L_2(\bu))=\emptyset$ and $\ell(\bq)\leq 2$.
             If $\ell(\bq)=1$, then $\bu\approx \bu+\bq$ is trivial.
             If $\ell(\bq)=2$, then $c(\bq)\subseteq c(L_2(\bu))$.
\end{itemize}
\end{lem}
\begin{pro}\label{pro43001}
$\mathsf{V}(S_{(4, 430)})$ is the ai-semiring variety defined by the identities
\begin{align}
xy& \approx yx; \label{43001}\\
x^4& \approx x^3; \label{43002}\\
x^3 &\approx x^2+x; \label{43003}\\
x+xy& \approx x+x^2y; \label{43004}\\
x^2+y^2&\approx x^2+y^2+xy; \label{43005} \\
x_1x_2x_3 & \approx x_1x_2x_3+x_1x_2x_3x_4. \label{43006}
\end{align}
\end{pro}
\begin{proof}
It is easy to check that $S_{(4, 430)}$ satisfies the identities (\ref{43001})--(\ref{43006}).
In the remainder it is enough to prove that every ai-semiring identity of $S_{(4, 430)}$
can be derived by (\ref{43001})--(\ref{43006}) and the identities defining $\mathbf{AI}$.
Let $\bu \approx \bu+\bq$ be such a nontrivial identity,
where $\bu=\bu_1+\bu_2+\cdots+\bu_n$ and $\bu_i, \bq \in X_c^+$, $1 \leq i \leq n$.
Since $S_2^0$ is isomorphic to $S_{(4, 430)}$, it follows that $S_2^0$ satisfies $\bu \approx \bu+\bq$
and so by Lemma \ref{lem001} $S_2$ satisfies $D_\bq(\bu) \approx D_\bq(\bu)+\bq$.
By Lemma \ref{lem201} we only need to consider the following three cases.

\textbf{Case 1.} $\ell(\bu_j)\geq 3$ for some $\bu_j\in D_\bq(\bu)$. Then
\[
\bu \approx \bu+\bu_j \stackrel{(\ref{43006})}\approx \bu+\bu_j+\bu_j\bq^3 \stackrel{(\ref{43001}), (\ref{43002})}{\approx} \bu+\bu_j+\bq^3\stackrel{(\ref{43003})}\approx \bu+\bq^2+\bq.
\]

\textbf{Case 2.} $c(L_1(D_\bq(\bu)))\cap c(L_2(D_\bq(\bu)))\neq\emptyset$.
Take $x$ in $c(L_1(D_\bq(\bu)))\cap c(L_2(D_\bq(\bu)))$.
Then $xy\in L_2(D_\bq(\bu))$ for some $y\in X$.
Now we have
\[
\bu \stackrel{(\ref{43001})}\approx \bu+x+xy \stackrel{(\ref{43004})}\approx \bu+x+x^2y.
\]
The remaining steps are similar to Case 1.

\textbf{Case 3.} $\ell(\bu_i)\leq 2$ for all $\bu_i \in D_\bq(\bu)$,
$c(L_1(D_\bq(\bu)))\cap c(L_2(D_\bq(\bu)))=\emptyset$ and $\ell(\bq)=2$.
Then
$c(\bq)\subseteq c(L_2(D_\bq(\bu)))$.
This implies that $L_1(D_\bq(\bu))$ is empty and so $D_\bq(\bu)=L_2(D_\bq(\bu))$.
Let $\bq=xy$. Then $D_\bq(\bu)=x^2+y^2$
and so
\[
\bu\approx \bu+x^2+y^2 \stackrel{(\ref{43005})}{\approx} \bu+x^2+y^2+xy \approx \bu+x^2+y^2+\bq.
\]
This derives the identity $\bu\approx \bu+\bq$.
\end{proof}

The following result can be found in \cite[Lemma 5.1]{rlzc}.
\begin{lem}\label{lem401}
Let $\bu\approx \bu+\bq$ be a nontrivial ai-semiring identity such that
$\bu=\bu_1+\bu_2+\cdots+\bu_n$ and $\bu_i, \bq \in X^+$, $1\leq i \leq n$.
Then $\bu\approx \bu+\bq$ is satisfied by $S_4$ if and only if $\bu$ and $\bq$ satisfy the following conditions:
\begin{itemize}
\item[$(1)$] $c(\bq)\subseteq  c(\bu)$;

\item[$(2)$] $\ell(\bu_i)\geq 2$ for some $\bu_i \in \bu$;

\item[$(3)$] If $\bu$ satisfies the property {\rm (T)}:
\[
(\forall \bu_i, \bu_j \in \bu) ~ m(t(\bu_i), \bu_j)\leq 1; m(t(\bu_i), \bu_j)= 1\Rightarrow t(\bu_i)=t(\bu_j),
\]
then $\bu+\bq$ also satisfies the property {\rm (T)} and so
\[
(\forall \bu_i \in \bu) ~ m(t(\bu_i), \bq)\leq 1; m(t(\bu_i), \bq)= 1\Rightarrow t(\bu_i)=t(\bq);
\]
\[
m(t(\bq), \bq)= 1; m(t(\bq), \bu_i)\leq 1; m(t(\bq), \bu_i)= 1\Rightarrow t(\bq)=t(\bu_i).
\]
\end{itemize}
\end{lem}

\begin{pro}\label{pro43401}
$\mathsf{V}(S_{(4, 434)})$ is the ai-semiring variety defined by the identities
\begin{align}
x^2y& \approx xy; \label{43401}\\
xyz & \approx yxz;  \label{43402}\\
(xy)^2&\approx x^2y^2; \label{43400} \\
x^2y^2&\approx y^2x^2; \label{43403} \\
x^2 &\approx x^2+x; \label{43404}\\
x+y^2 &\approx x+y^2+{(xy)}^2; \label{43405}\\
xz+yz &\approx xz+yz+xyz; \label{43406}\\
x_1x_2+x_2x_3+x_4 &\approx x_1x_2+x_2x_3+x_4+x_1x_2x_3x_4, \label{43408}
\end{align}
where $x_1$ may be empty in $(\ref{43408})$.
\end{pro}
\begin{proof}
It is easy to verify that $S_{(4, 434)}$ satisfies the identities (\ref{43401})--(\ref{43408}).
In the remainder it is enough to prove that every ai-semiring identity of $S_{(4, 434)}$
can be derived by (\ref{43401})--(\ref{43408}) and the identities defining $\mathbf{AI}$.
Let $\bu \approx \bu+\bq$ be such a nontrivial identity,
where $\bu=\bu_1+\bu_2+\cdots+\bu_n$ and $\bu_i, \bq\in X^+$, $1 \leq i \leq n$.
Since $S_4^0$ is isomorphic to $S_{(4, 434)}$, it follows that
$S_4$ satisfies $D_\bq(\bu) \approx D_\bq(\bu)+\bq$.
By Lemma \ref{lem401} we have that $c(\bq)=c(D_\bq(\bu))$ and $\ell(\bu_r)\geq 2$ for some $\bu_r\in D_\bq(\bu)$.
Also, if $D_\bq(\bu)$ satisfies the property {\rm (T)},
then $D_\bq(\bu)+\bq$ satisfies the property {\rm (T)}.
Consider the following two cases.

\textbf{Case 1.}
$D_\bq(\bu)$ satisfies the property {\rm (T)}.
Then $D_\bq(\bu)+\bq$ also satisfies the property {\rm (T)}.
This implies that $\ell(\bq)\geq 2$, $m(t(\bq), \bq)=1$, $t(D_\bq(\bu))=t(\bq)$ and so
$c(L_{\geq 2}(D_\bq(\bu)))=c(\bq)$.
Let $L_{\geq 2}(D_\bq(\bu))=\bu_1+\bu_2+\cdots+\bu_k$. We have
\begin{align*}
\bu
&\approx \bu+L_{\geq 2}(D_\bq(\bu))\\
&\approx \bu+\bu_1+\bu_2+\cdots+\bu_k \\
&\approx \bu+p(\bu_1)t(\bq)+p(\bu_2)t(\bq)+\cdots+p(\bu_k)t(\bq) \\
&\approx \bu+\bu_1+\bu_2+\cdots+\bu_k+p(\bu_1)p(\bu_2)\cdots p(\bu_k)t(\bq) &&(\text{by}~(\ref{43406})) \\
&\approx  \bu+\bu_1+\bu_2+\cdots+\bu_k+\bq.&&(\text{by}~(\ref{43401}), (\ref{43402}))
\end{align*}

\textbf{Case 2.} $D_\bq(\bu)$ does not satisfy the property {\rm (T)}.
Then either $m(t(\bu_i), \bu_j)\geq 2$ for some $\bu_i, \bu_j \in D_\bq(\bu)$
or $m(t(\bu_i), \bu_j)=1$ and $t(\bu_i) \neq t(\bu_j)$ for some $\bu_i, \bu_j \in D_\bq(\bu)$.
If $m(t(\bu_i), \bu_j)\geq 2$ for some $\bu_i, \bu_j \in D_{\bq}(\bu)$,
then $t(\bu_i) \in c(p(\bu_j))$.
By (\ref{43401}) and (\ref{43402}) we can deduce
$\bu_j\approx \bp{t(\bu_i)}^2$ or $\bu_j\approx {t(\bu_i)}^2\bp$ for some word $\bp$,
where $c(\bp)\subseteq c(\bu_j)$.
We write $D_\bq(\bu)=\bu_1+\bu_2+\cdots+\bu_s$.

If $\bu_j\approx \bp{t(\bu_i)}^2$ is obtained, then
\begin{align*}
\bu
&\approx \bu+D_\bq(\bu)\\
&\approx \bu+\bu_1+\bu_2+\cdots+\bu_s+\bu_j \\
&\approx \bu+\bu_1+\bu_2+\cdots+\bu_s+\bp{t(\bu_i)}^2 \\
&\approx \bu+\bu_1+\bu_2+\cdots+\bu_s+(\bu_1\bu_2\cdots \bu_s\bp t(\bu_i))^2 &&(\text{by}~(\ref{43401}), (\ref{43400}),(\ref{43405})) \\
&\approx  \bu+\bu_1+\bu_2+\cdots+\bu_s+\bq^2&&(\text{by}~(\ref{43401}), (\ref{43400}),(\ref{43403}))\\
&\approx  \bu+\bu_1+\bu_2+\cdots+\bu_s+\bq^2+\bq.&&(\text{by}~ (\ref{43404}))
\end{align*}

If $\bu_j\approx{t(\bu_i)}^2\bp$, then
\begin{align*}
\bu
&\approx \bu+\bu_i+\bu_j+D_\bq(\bu) \\
&\approx \bu+p(\bu_i)t(\bu_i)+{t(\bu_i)}^2\bp+D_\bq(\bu)\\
&\approx \bu+p(\bu_i)t(\bu_i)+{t(\bu_i)}^2\bp+D_\bq(\bu)+\bw &&(\text{by}~ (\ref{43401}), (\ref{43408})) \\
&\approx  \bu+p(\bu_i)t(\bu_i)+{t(\bu_i)}^2\bp+D_\bq(\bu)+\bq^2 &&(\text{by}~(\ref{43402}), (\ref{43400}),(\ref{43403}))\\
&\approx  \bu+p(\bu_i)t(\bu_i)+{t(\bu_i)}^2\bp+D_\bq(\bu)+\bq^2+\bq, &&(\text{by}~ (\ref{43404}))
\end{align*}
where $\bw=p(\bu_i)^2t(\bu_i)^2\bp^2\bu_1^2\bu_2^2\cdots \bu_s^2$.

Finally, if $m(t(\bu_i), \bu_j)=1$ and $t(\bu_i) \neq t(\bu_j)$ for some $\bu_i, \bu_j \in D_{\bq}(\bu)$.
Then $m(t(\bu_i), p(\bu_j))=1$ and so $\bu_j \approx{t(\bu_i)}\bp$ for some nonempty word $\bp$ is derived.
The remaining steps are similar to the preceding case.
\end{proof}

It is easy to see that $S_{(4, 432)}$ and $S_{(4, 434)}$  have dual multiplications. By Proposition $\ref{pro43401}$ we immediately deduce
\begin{cor}
The ai-semiring $S_{(4, 432)}$ is finitely based.
\end{cor}

\begin{lem}\label{lem4401}
Let $\bu\approx \bu+\bq$ be a nontrivial ai-semiring identity such that
$\bu=\bu_1+\bu_2+\cdots+\bu_n$ and $\bu_i, \bq \in X^+$, $1\leq i \leq n$.
Suppose that $\bu\approx \bu+\bq$ is satisfied by $S_{44}$.
Then $\ell(\bq)\geq 2$ and $D_\bq(\bu)\neq\emptyset$.
If $M_{1}(\bq)\neq\emptyset$,
then any $x \in M_{1}(\bq)$, there exists $\bu_i \in D_\bq(\bu)$ such that $m(x, \bu_i)\leq 1$.
\end{lem}
\begin{proof}
Suppose that $S_{44}$ satisfies $\bu\approx \bu+\bq$.
It is easy to see that $N_2$ is isomorphic to $\{1, 2\}$ and so $N_2$ satisfies $\bu \approx \bu+\bq$.
This implies that $\ell(\bq)\geq 2$.
Since $D_2$ is isomorphic to $\{2, 3\}$, it follows that $D_2$ satisfies $\bu \approx \bu+\bq$
and so $c(\bu_j) \subseteq c(\bq)$ for some $\bu_j\in \bu$. So $D_\bq(\bu)$ is nonempty.
Suppose that $M_{1}(\bq)$ is nonempty.
If there exists $x\in M_{1}(\bq)$ such that for any $\bu_i \in D_\bq(\bu)$, $ m(x, \bu_i)\geq 2$,
then we consider the semiring homomorphism $\varphi: P_f(X^+) \to S_{44}$ defined by
$\varphi(x)=1$, $\varphi(y)=3$ if $y\in c(\bq)\setminus\{x\}$, and $\varphi(y)=2$ otherwise.
It is easy to see that $\varphi(\bu)=2$ and $\varphi(\bq)=1$, and so
$\varphi(\bu)\neq \varphi(\bu+\bq)$,
a contradiction.
Thus any $x \in M_{1}(\bq)$, there exists $\bu_i \in D_\bq(\bu)$ such that $m(x, \bu_i)\leq 1$.
\end{proof}

\begin{pro}\label{pro24082701}
$\mathsf{V}(S_{(4, 447)})$ is the ai-semiring variety defined by the identities
\begin{align}
x^3 &\approx x^2; \label{44702}\\
xy & \approx yx; \label{44701}\\
xy &\approx xy+xyz; \label{44703}\\
xy &\approx x^2y+xy^2; \label{44704}\\
y+y^2z &\approx y+y^2z+xyz, \label{44706}
\end{align}
where $z$ may be empty in $(\ref{44706})$.
\end{pro}
\begin{proof}
It is easy to check that $S_{(4, 447)}$ satisfies the identities (\ref{44702})--(\ref{44706}).
In the remainder it is enough to prove that every ai-semiring identity of $S_{(4, 447)}$
is derivable from (\ref{44702})--(\ref{44706}) and the identities defining $\mathbf{AI}$.
Let $\bu \approx \bu+\bq$ be such a nontrivial identity,
where $\bu=\bu_1+\bu_2+\cdots+\bu_n$ and $\bu_i, \bq\in X^+$, $1 \leq i \leq n$.
Since $S_{44}$ is isomorphic to $\{1, 2, 4\}$, it follows that $S_{44}$ satisfies $\bu \approx \bu+\bq$.
By Lemma \ref{lem4401} we have that $\ell(\bq)\geq 2$.
It is easy to see that $T_2^0$ is isomorphic to $\{1, 2, 3\}$
and so $T_2^0$ satisfies $\bu \approx \bu+\bq$.
By Lemma \ref{lem001} we have that $L_{\geq 2}(\bu)\cap D_{\bq}(\bu)$ is nonempty.

\textbf{Case 1.}
$M_1(\bq)$ is empty. Then $m(x, \bq)\geq 2$ for all $x\in c(\bq)$.
Take $\bu_i$ in $L_{\geq 2}(\bu)\cap D_{\bq}(\bu)$.
Then
\[
\bu \approx \bu+\bu_i\stackrel{(\ref{44703})}\approx \bu+\bu_i+\bu_i\bq \stackrel{(\ref{44701}), (\ref{44702})}\approx \bu+\bu_i+\bq.
\]
This derives $\bu \approx \bu+\bq$.

\textbf{Case 2.}
$M_1(\bq)$ is nonempty. Let
$c(\bq)=\{x_1, \ldots, x_r, y_1, \ldots, y_s\}$,
where $m(x_i, \bq )\geq 2$, $m(y_j, \bq )=1$, $0 \leq i \leq r$, $1 \leq j \leq s$.
Then
\begin{align*}
\bq
&\approx x_1^2x_2^2\cdots x_r^2y_1y_2\cdots y_s &&(\text{by}~ (\ref{44702}), (\ref{44701}))\\
&\approx\sum\limits_{1\leq j\leq s}x_1^2x_2^2\cdots x_r^2y_1^2\cdots y_{j-1}^2y_{j+1}^2\cdots y_s^2y_j.
&&(\text{by}~ (\ref{44702}), (\ref{44701}), (\ref{44704}))
\end{align*}
So we only need to consider the case that $\bq=x_1^2x_2^2\cdots x_k^2y$.
We shall show that there exists $\bu_j \in D_\bq(\bu)$ such that
either $ m(y, \bu_j)\leq 1$, $\ell(\bu_j)\geq 2$ or $\bu_j=y$.
Suppose that this is not true.
Let us consider the semiring homomorphism $\varphi: P_f(X^+) \to S_{(4, 447)}$ defined by
$\varphi(y)=4$, $\varphi(x)=3$ if $x\in c(\bq)\setminus{\{y\}}$, and $\varphi(x)=2$ otherwise.
It is easy to see that $\varphi(\bu)=3$ or $2$ and $\varphi(\bq)=4$, a contradiction.
So there exists $\bu_j \in D_\bq(\bu)$ such that
either $ m(y, \bu_j)\leq 1$, $\ell(\bu_j)\geq 2$ or $\bu_j=y$.

\textbf{Subcase 2.1.} There exists $\bu_j \in D_\bq(\bu)$ such that
$ m(y, \bu_j)\leq 1$, $\ell(\bu_j)\geq 2$. Now we have
\[
\bu \approx \bu+\bu_j\stackrel{(\ref{44703})}\approx \bu+\bu_j+\bu_j\bp\stackrel{(\ref{44701}), (\ref{44702})}\approx \bu+\bu_j+\bq,
\]
where $\bp=p(\bq)$ if $y\in c(\bu_j)$, and $\bp=\bq$ otherwise.

\textbf{Subcase 2.2.} There exists $\bu_j \in D_\bq(\bu)$ such that $\bu_j=y$.
Choose $\bu_i$ in $L_{\geq 2}(\bu)\cap D_{\bq}(\bu)$.
If $ m(y, \bu_i)\leq 1$, then
\[
\bu \approx \bu+\bu_i\stackrel{(\ref{44703})}\approx \bu+\bu_i+\bu_i\bp\stackrel{(\ref{44701}), (\ref{44702})}\approx \bu+\bu_i+\bq,
\]
where $\bp=p(\bq)$ if $y\in c(\bu_i)$, and $\bp=\bq$ otherwise.
If $m(y, \bu_i)\geq 2$, then (\ref{44701}) and (\ref{44702}) imply $\bu_i\approx y^2\bp$ for some word $\bp$,
where $c(\bp)\subseteq c(\bu_i)\setminus\{y\}$.
Furthermore, we have
\begin{align*}
\bu
&\approx \bu+\bu_j+\bu_i  \\
&\approx \bu+y+y^2\bp \\
&\approx \bu+y+y^2\bp+p(\bq)y\bp&&(\text{by}~ (\ref{44706}))\\
&\approx \bu+y+y^2\bp+\bq.&&(\text{by}~ (\ref{44701}), (\ref{44702}))
\end{align*}
This completes the proof.
\end{proof}

\begin{pro}\label{pro45001}
$\mathsf{V}(S_{(4, 450)})$ is the ai-semiring variety defined by the identities
\begin{align}
x^3 &\approx x^2; \label{45002}\\
xy & \approx yx; \label{45001}\\
xy &\approx x^2y+xy^2; \label{45003}\\
x+x^2y &\approx x+xy; \label{45005}\\
x+y &\approx x+y+xy. \label{45004}
\end{align}
\end{pro}
\begin{proof}
It is easily verified that $S_{(4, 450)}$ satisfies the identities (\ref{45002})--(\ref{45004}).
In the remainder we need only prove that every ai-semiring identity of $S_{(4, 450)}$
can be derived by (\ref{45002})--(\ref{45004}) and the identities defining $\mathbf{AI}$.
Let $\bu \approx \bu+\bq$ be such a nontrivial identity,
where $\bu=\bu_1+\bu_2+\cdots+\bu_n$ and $\bu_i, \bq\in X^+$, $1 \leq i \leq n$.
Since $M_2^0$ is isomorphic to $\{1, 2, 3\}$,
it follows that $M_2^0$ satisfies $\bu \approx \bu+\bq$ and so $c(\bq)=c(D_\bq(\bu))$.
Let $D_\bq(\bu)=\bu_1+\bu_2+\cdots+\bu_k$. Then $c(\bq)=c(\bu_1\bu_2\cdots\bu_k)$.
Since $S_{44}$ is isomorphic to $\{1, 2, 4\}$, we have that $S_{44}$ satisfies $\bu \approx \bu+\bq$
and so by Lemma \ref{lem4401} $\ell(\bq)\geq 2$.

\textbf{Case 1.} $M_{1}(\bq)$ is empty.
Then $m(x, \bq)\geq 2$ for all $x\in c(\bq)$ and so
\[
\bu \approx \bu+D_\bq(\bu)\stackrel{(\ref{45004})}\approx \bu+\bu_1^2\bu_2^2\cdots\bu_k^2
\stackrel{(\ref{45002}), (\ref{45001})}\approx \bu+\bq.
\]

\textbf{Case 2.} $M_{1}(\bq)$ is nonempty.
By a similar argument in Case 2 of the proof of Proposition \ref{pro24082701},
it is enough to consider the case that $\bq=x_1^2x_2^2\cdots x_r^2y$.
We shall show that there exists $\bu_j \in D_\bq(\bu)$ such that $m(y, \bu_j)=1$.
Suppose that this is not true.
Consider the semiring homomorphism $\varphi: P_f(X^+) \to S_{(4, 450)}$ defined by
$\varphi(y)=4$, $\varphi(x)=3$ if $x\in c(\bq)\setminus \{y\}$,
and $\varphi(x)=2$ otherwise.
It is easy to see that $\varphi(\bu)=3$ and $\varphi(\bq)=4$, a contradiction.
Thus there exists $\bu_j \in D_\bq(\bu)$ such that $m(y, \bu_j)=1$.
Furthermore, we have
\begin{align*}
\bu
&\approx \bu+\bu_1^2\bu_2^2\cdots\bu_k^2 && (\text{by}~ (\ref{45004}))  \\
&\approx \bu+\bu_j^2x_1^2x_2^2\cdots x_r^2 && (\text{by}~ (\ref{45002}), (\ref{45001})) \\
&\approx \bu+\bu_j+\bu_j^2x_1^2x_2^2\cdots x_r^2\\
&\approx \bu+\bu_j+\bu_jx_1^2x_2^2\cdots x_r^2 && (\text{by}~ (\ref{45005})) \\
&\approx \bu+\bu_j+\bq. && (\text{by}~ (\ref{45002}), (\ref{45001}))
\end{align*}
This derives the identity $\bu\approx \bu+\bq$.
\end{proof}

\begin{lem}\label{lem4601}
Let $\bu\approx \bu+\bq$ be a nontrivial ai-semiring identity such that
$\bu=\bu_1+\bu_2+\cdots+\bu_n$ and $\bu_i, \bq \in X^+$, $1\leq i \leq n$.
Suppose that $\bu\approx \bu+\bq$ is satisfied by $S_{46}$.
Then $\ell(\bq)\geq 2$.
If $ m(t(\bq), \bq)= 1$,
then there exists $\bu_i \in D_\bq(\bu)$ such that
$t(\bq)\notin c(p(\bu_i))$.
\end{lem}
\begin{proof}
Suppose that $S_{46}$ satisfies $\bu\approx \bu+\bq$.
Since $N_2$ is isomorphic to $\{1, 2\}$, it follows that $N_2$ satisfies $\bu \approx \bu+\bq$
and so $\ell(\bq)\geq 2$.
It is easy to see that $D_2$ is isomorphic to $\{2, 3\}$ and so $D_2$ satisfies $\bu \approx \bu+\bq$.
This implies that $D_\bq(\bu)$ is nonempty.
Let $m(t(\bq), \bq)=1$.
Suppose by way of contradiction that for any $\bu_i \in D_\bq(\bu)$,
$t(\bq)\in c(p(\bu_i))$.
Consider the semiring homomorphism $\varphi: P_f(X^+) \to S_{46}$ defined by
$\varphi(t(\bq))=1$, $\varphi(x)=3$ if $x\in c(\bq)\setminus\{t(\bq)\}$,
and $\varphi(x)=2$ otherwise.
Then $\varphi(\bu)=2$ and $\varphi(\bq)=1$, and so
$\varphi(\bu)\neq \varphi(\bu+\bq)$,
a contradiction.
So there exists $\bu_i \in D_\bq(\bu)$ such that $t(\bq)\notin c(p(\bu_i))$.
\end{proof}

\begin{pro}\label{pro44201}
$\mathsf{V}(S_{(4, 442)})$ is the ai-semiring variety defined by the identities
\begin{align}
xyz & \approx yxz; \label{44201}\\
x^2y &\approx xy; \label{44202}\\
x_1x_2 &\approx x_1x_2+x_3x_1x_2x_4; \label{44203}\\
x_1+x_2x_1^2x_3&\approx x_1+x_2x_1^2x_3+x_4x_3x_2x_1, \label{id24082601}
\end{align}
where
$x_3$ and $x_4$ may be empty in $(\ref{44203})$, $x_2$ and $x_3$ may be empty in $(\ref{id24082601})$.
\end{pro}
\begin{proof}
It is easy to check that $S_{(4, 442)}$ satisfies the identities (\ref{44201})--(\ref{id24082601}).
In the remainder we need only prove that every ai-semiring identity of $S_{(4, 442)}$
can be derived by (\ref{44201})--(\ref{id24082601}) and the identities defining $\mathbf{AI}$.
Let $\bu \approx \bu+\bq$ be such a nontrivial identity,
where $\bu=\bu_1+\bu_2+\cdots+\bu_n$ and $\bu_i, \bq\in X^+$, $1 \leq i \leq n$.
Since $S_{46}$ is isomorphic to $\{1, 2, 4\}$, it follows that
$S_{46}$ satisfies $\bu \approx \bu+\bq$. By Lemma \ref{lem4601} we obtain that $\ell(\bq)\geq 2$.
It is easy to see that $T_2^0$ is isomorphic to $\{1, 2, 3\}$ and so $T_2^0$ satisfies $\bu \approx \bu+\bq$.
By Lemma \ref{lem001} we have that $L_{\geq 2}(\bu)\cap D_{\bq}(\bu)$ is nonempty.
So there exists $\bu_i \in \bu$ such that $\ell(\bu_i)\geq 2$ and $c(\bu_i)\subseteq c(\bq)$.

\textbf{Case 1.} $ m(t(\bq), \bq)\geq 2$. Then
\[
\bu \approx \bu+\bu_i\stackrel{(\ref{44203})}\approx \bu+\bu_i+\bu_i\bq\stackrel{(\ref{44201}),(\ref{44202})}\approx \bu+\bu_i+\bq.
\]

\textbf{Case 2.} $m(t(\bq), \bq)=1$.
If $m(t(\bq), \bu_i)=0$, then
 \[
\bu \approx \bu+\bu_i\stackrel{(\ref{44203})}\approx \bu+\bu_i+\bu_i\bq\stackrel{(\ref{44201}),(\ref{44202})}\approx \bu+\bu_i+\bq.
\]
If $m(t(\bq), \bu_i)=1$ and $t(\bq)=t(\bu_i)$, we have
\[
\bu \approx \bu+\bu_i\stackrel{(\ref{44203})}\approx \bu+\bu_i+p(\bq)\bu_i\stackrel{(\ref{44201}),(\ref{44202})}\approx \bu+\bu_i+\bq.
\]
If $m(t(\bq), \bu_i)\geq 2$ or $m(t(\bq), \bu_i)=1$, $t(\bq)\neq t(\bu_i)$,
then (\ref{44201}) and (\ref{44202}) imply the identity
\begin{equation}\label{id24082610}
\bu_i \approx \bp_1t(\bq)^2\bp_2
\end{equation}
for some $\bp_1, \bp_2 \in X^*$,
where $c(\bp_1), c(\bp_2)\subseteq c(\bu_i)$ and $t(\bq)\notin c(\bp_1)\cup c(\bp_2)$.

\textbf{Subcase 2.1.} $m(t(\bq), \bu_j)=1$ and $t(\bq)=t(\bu_j)$ for some $\bu_j \in D_\bq(\bu)$.
If $\ell(\bu_j)\geq 2$, then
\[
\bu \approx \bu+\bu_j\stackrel{(\ref{44203})}\approx \bu+\bu_j+p(\bq)\bu_j\stackrel{(\ref{44201}),(\ref{44202})}\approx \bu+\bu_j+\bq.
\]
If $\ell(\bu_j)=1$, then $\bu_j=t(\bq)$ and so
\begin{align*}
\bu
&\approx \bu+\bu_j+\bu_i\\
&\approx \bu+t(\bq)+\bp_1t(\bq)^2\bp_2 &&(\text{by}~ (\ref{id24082610})) \\
&\approx \bu+t(\bq)+\bp_1t(\bq)^2\bp_2+p(\bq)\bp_1\bp_2t(\bq) &&(\text{by}~ (\ref{id24082601}))\\
&\approx \bu+t(\bq)+\bp_1t(\bq)^2\bp_2+\bq. &&(\text{by}~ (\ref{44201}),(\ref{44202}))
\end{align*}

\textbf{Subcase 2.2.} $m(t(\bq), \bu_j)\neq1$ or $t(\bq)\neq t(\bu_j)$ for all $\bu_j \in D_\bq(\bu)$.
We shall prove that $t(\bq)\notin c(\bu_k)$ and $\ell(\bu_k)\geq 2$ for some $\bu_k \in D_\bq(\bu)$.
Indeed, suppose that this is not true. Then for any $\bu_k \in D_\bq(\bu)$, $t(\bq)\in c(\bu_k)$ or $\ell(\bu_k)=1$.
Let $\varphi: P_f(X^+) \to S_{(4, 442)}$ be a homomorphism such that
$\varphi(t(\bq))=4$, $\varphi(x)=3$ if $x\in c(p(\bq))$ and $\varphi(x)=2$ otherwise.
It is easy to see that $\varphi(\bu)=3$ and $\varphi(\bq)=4$,
and so $\varphi(\bu) \neq \varphi(\bu+\bq)$, a contradiction.
So there exists $\bu_k \in D_\bq(\bu)$ such that $t(\bq)\notin c(\bu_k)$ and $\ell(\bu_k)\geq 2$.
We deduce
\[
\bu \approx \bu+\bu_k\stackrel{(\ref{44203})}\approx \bu+\bu_k+\bu_k\bq\stackrel{(\ref{44201}),(\ref{44202})}\approx \bu+\bu_k+\bq.
\]
This derives $\bu \approx \bu+\bq$.
\end{proof}

It is easy to see that $S_{(4, 446)}$ and $S_{(4, 442)}$ have dual multiplications.
By Proposition $\ref{pro44201}$ we immediately deduce
\begin{cor}
The ai-semiring $S_{(4, 446)}$ is finitely based.
\end{cor}

\begin{pro}\label{pro44301}
$\mathsf{V}(S_{(4, 443)})$ is the ai-semiring variety defined by the identities
\begin{align}
xyz & \approx yxz; \label{44301}\\
x^2y &\approx xy; \label{44302}\\
x &\approx x+yx. \label{44303}
\end{align}
\end{pro}
\begin{proof}
It is easy to verify that $S_{(4, 443)}$ satisfies the identities (\ref{44301})--(\ref{44303}).
In the remainder we need only prove that every ai-semiring identity of $S_{(4, 443)}$
is derivable from (\ref{44301})--(\ref{44303}) and the identities defining $\mathbf{AI}$.
Let $\bu \approx \bu+\bq$ be such a nontrivial identity,
where $\bu=\bu_1+\bu_2+\cdots+\bu_n$ and $\bu_i, \bq \in X^+$, $1 \leq i \leq n$.
Since $R_2^0$ is isomorphic to $\{1, 2, 3\}$, it follows that $R_2^0$ satisfies $\bu \approx \bu+\bq$
and so there exists $\bu_i \in \bu$ such that $c(\bu_i) \subseteq c(\bq)$ and $t(\bu_i)=t(\bq)$.
It is easy to check that $S_{46}$ is isomorphic to $\{1, 2, 4\}$ and so $S_{46}$ also satisfies $\bu \approx \bu+\bq$.
By Lemma \ref{lem4601} we have that $\ell(\bq)\geq 2$. Consider the following two cases.

\textbf{Case 1.} $ m(t(\bq), \bq)\geq 2$. Then
\[
\bu \approx \bu+\bu_i\stackrel{(\ref{44303})}\approx \bu+\bu_i+\bq\bu_i\stackrel{(\ref{44301}),(\ref{44302})}\approx \bu+\bu_i+\bq.
\]
So we obtain $\bu \approx \bu+\bq$.

\textbf{Case 2.} $ m(t(\bq), \bq)=1$.
Let $c(p(\bq))=\{x_1, x_2, \ldots, x_m\}$.
By the identities (\ref{44301}) and (\ref{44302}) we can deduce the identity
\begin{equation}\label{id24090301}
\bq \approx x_1^2x_2^2\cdots x_m^2t(\bq).
\end{equation}
We shall show that there exists $\bu_j \in D_\bq(\bu)$ such that $m(t(\bq), \bu_j)=1$ and $t(\bq)=t(\bu_j)$.
Indeed, suppose that this is not true.
Consider the semiring homomorphism $\varphi: P_f(X^+) \to S_{(4, 443)}$ defined by
$\varphi(t(\bq))=4$, $\varphi(x)=3$ for every $x\in c(p(\bq))$ and $\varphi(x)=2$ otherwise.
Then $\varphi(\bu)=3$, $\varphi(\bq)=4$ and so $\varphi(\bu)\neq\varphi(\bu+\bq)$, a contradiction.
So there exists $\bu_j \in D_\bq(\bu)$ such that $m(t(\bq), \bu_j)=1$ and $t(\bq)=t(\bu_j)$.
Furthermore, we have
\[
\bu \approx \bu+\bu_j\stackrel{(\ref{44303})}\approx \bu+\bu_j+p(\bq)\bu_j
\stackrel{(\ref{44301}),(\ref{44302})}\approx \bu+\bu_j+x_1^2x_2^2\cdots x_m^2t(\bq)\stackrel{(\ref{id24090301})}\approx \bu+\bu_j+\bq.
\]
This derives $\bu \approx \bu+\bq$.
\end{proof}

\begin{cor}
The ai-semiring $S_{(4, 456)}$ is finitely based.
\end{cor}
\begin{proof}
It is easy to see that $S_{(4, 456)}$ and $S_{(4, 443)}$ have dual multiplications.
By Proposition $\ref{pro44301}$ we obtain that $S_{(4, 456)}$ is finitely based.
\end{proof}

\section{Equational basis of $S_{(4, 427)}$ and $S_{(4, 428)}$}
In this section we shall show that $S_{(4, 427)}$ and $S_{(4, 428)}$
are both finitely based. It is easy to see that these two semirings contain copies of $S_{41}$.
Let $\bp$ be a word in $X^+$ and $Y$ a subset of $X$.
Then $i(\bp)$ denotes the the word that is obtained from $\bp$ by
retaining the first occurrence of each variable, and
$h_Y(\bp)$ denotes the first letter in the word
that is obtained from $\bp$ by deleting all variables in $Y$.
The following result, which is due to \cite[Lemma 2.4]{gpz}, provides a solution of the equational problem for $S_{41}$.

\begin{lem}\label{lem4101}
Let $\bu\approx \bu+\bq$ be an ai-semiring identity such that
$\bu=\bu_1+\bu_2+\cdots+\bu_n$ and $\bu_i, \bq \in X^+$, $1\leq i \leq n$.
Then $\bu\approx \bu+\bq$ is satisfied by $S_{41}$ if and only if
for any $Y \subseteq c(\bq)$, there exists $\bu_j\in \bu$ such that $h_Y(\bu_j)=h_Y(\bq)$.
\end{lem}

\begin{pro}\label{pro42701}
$\mathsf{V}(S_{(4, 427)})$ is the ai-semiring variety defined by the identities
\begin{align}
x^3 &\approx x^2; \label{42701}\\
xyx &\approx xy; \label{42702}\\
x^2y &\approx xy;  \label{id24082725}\\
xy^2 &\approx xy; \label{42703}\\
xy & \approx (xy)^2; \label{42704}\\
x^2 &\approx x^2+x; \label{42705}\\
x+xy &\approx x^2+xy; \label{42706}\\
x^2+yz & \approx x^2+yz+yx; \label{42707}\\
x+yxz &\approx x+yxz+yx; \label{42708}\\
x_1+x_2+x_3x_2x_4 &\approx x_1+x_2+x_3x_2x_4+x_3x_2x_1x_4. \label{42709}
\end{align}
\end{pro}
\begin{proof}
It is easy to check that $S_{(4, 427)}$ satisfies the identities (\ref{42701})--(\ref{42709}).
In the remainder it is enough to show that every ai-semiring identity of $S_{(4, 427)}$
is derivable from (\ref{42701})--(\ref{42709}) and the identities defining $\mathbf{AI}$.
Let $\bu \approx \bu+\bq$ be such a nontrivial identity,
where $\bu=\bu_1+\bu_2+\cdots+\bu_n$ and $\bu_i, \bq\in X^+$, $1 \leq i \leq n$.
Since $D_2$ is isomorphic to $\{2, 4\}$, it follows that $D_2$ satisfies $\bu \approx \bu+\bq$
and so $D_{\bq}(\bu)$ is nonempty. So there exists $\bu_{i_1} \in \bu$ such that $c(\bu_{i_1})\subseteq c(\bq)$.
It is easy to see that $S_{58}$ is isomorphic to $\{1, 2, 3\}$
and so $S_{58}$ satisfies $\bu \approx \bu+\bq$.
If $\ell(\bq)\geq 2$, then by Lemma \ref{lem5801} $L_{\geq 2}(\bu)\cap H_{\bq}(\bu)$ is nonempty
and so $\ell(\bu_k)\geq 2$ and $h(\bu_k)=h(\bq)$ for some $\bu_k \in \bu$.

Let $|c(\bq)|=1$.
If $\ell(\bq)=1$, then $\bu_{i_1}=\bq^\ell$ for some $\ell\geq 2$. So we have
\[
\bu\approx \bu+\bu_{i_1} \approx \bu+\bq^\ell \stackrel{(\ref{42705})}\approx \bu+\bq^\ell+\bq.
\]
Now assume that $\ell(\bq)\geq 2$. Then $\bq=x^r$ for some $x \in X$ and some $r \geq 2$.
This implies that $\bu_{i_1}=x^s$ for some $s \geq 1$ and $\bu_k=xs(\bu_k)$, where $s(\bu_k)$ is nonempty.
If $s=1$, then $\bu_{i_1}=x$ and so
\[
\bu\approx \bu+x+xs(\bu_k)\stackrel{(\ref{42706})}\approx \bu+x^2+xs(\bu_k)\stackrel{(\ref{42701})}
\approx \bu+x^r+xs(\bu_k)\approx \bu+\bq+xs(\bu_k).
\]
If $s\geq 2$, then
\[
\bu \approx \bu+\bu_{i_1}\approx\bu+x^s\stackrel{(\ref{42701})}\approx\bu+x^r\approx\bu+\bq.
\]
This implies the identity $\bu\approx \bu+\bq$.

Now let $|c(\bq)|\geq 2$.
By (\ref{42701})--(\ref{42703}) we deduce the identity $\bq \approx i(\bq)$
and so $\bu \approx \bu+i(\bq)$ is satisfied by $S_{(4, 427)}$.
Let $i(\bq)=x_1x_2\cdots x_m$, $m\geq 2$.
Since $S_{41}$ is isomorphic to $\{1, 2, 4\}$, we have that $S_{41}$ satisfies $\bu \approx \bu+i(\bq)$.
By Lemma \ref{lem4101} it follows that for any $2\leq j\leq m$ and $Y_j=\{x_1, x_2, \ldots ,x_{j-1}\}$,
there exists $\bu_j \in \bu$
such that $\bu_j=\bu_{j_1}x_j\bu_{j_2}$ for some $\bu_{j_1}, \bu_{j_2} \in X^*$, where $c(\bu_{j_1})\subseteq Y_j$.

\textbf{Case 1.}
$D_\bq(\bu)\cap L_{\geq 2}(\bu)$ is nonempty.
Choose a word $\bp$ in $D_\bq(\bu)\cap L_{\geq 2}(\bu)$.
We shall show by induction on $j$ that $\bu \approx \bu+x_1x_2\cdots x_j\bp$
is derivable from (\ref{42701})--(\ref{42709}), $1 \leq j\leq m$.
If $j=1$, then
\[
\bu \approx \bu+\bp+\bu_k \stackrel{(\ref{42704})}\approx \bu+\bp^2+x_1s(\bu_k) \stackrel{(\ref{42707})}\approx \bu+\bp^2+x_1s(\bu_k)+x_1\bp.
\]
This implies $\bu \approx \bu+x_1\bp$.
Let $2\leq j\leq m$.
Suppose that $\bu\approx \bu+x_1x_2\cdots x_{j-1}\bp$ is derivable from (\ref{42701})--(\ref{42709}).
If $h(\bp)=x_j$, then
\[
\bu\approx \bu+x_1x_2\cdots x_{j-1}\bp \stackrel{(\ref{id24082725})}\approx \bu+x_1x_2\cdots x_{j-1}x_j\bp.
\]
Now assume that $h(\bp)\neq x_j$.

\textbf{Subcase 1.1.} $\ell(\bu_j)\geq 2 $ and $h_{Y_j}(\bq)=h_{Y_j}(\bu_j)$ for some $\bu_j \in \bu $.
Then $\bu_{j}=\bu_{j_1}x_j\bu_{j_2}$,
where $c(\bu_{j_1})\subseteq Y_j$,
$\bu_{j_1}$ and $\bu_{j_2}$ can not be empty words simultaneously.
Now we have
\begin{align*}
\bu
&\approx \bu+\bu_j+x_1x_2\cdots x_{j-1}\bp  \\
&\approx \bu+\bu_j^2+x_1x_2\cdots x_{j-1}\bp &&(\text{by}~(\ref{42704}))\\
&\approx \bu+(\bu_{j_1}x_j\bu_{j_2})^2+x_1x_2\cdots x_{j-1}\bp\\
&\approx \bu+(\bu_{j_1}x_j\bu_{j_2})^2+x_1x_2\cdots x_{j-1}\bp+x_1x_2\cdots x_{j-1}\bu_{j_1}x_j\bu_{j_2} &&(\text{by}~(\ref{42707}))\\
&\approx \bu+(\bu_{j_1}x_j\bu_{j_2})^2+x_1x_2\cdots x_{j-1}\bp+x_1x_2\cdots x_{j-1}x_j\bu_{j_2}.
&&(\text{by}~(\ref{42702}), (\ref{id24082725}), (\ref{42703}))
\end{align*}
This implies the identity
\[
\bu\approx \bu+x_1x_2\cdots x_{j-1}x_j\bu_{j_2}.
\]
Furthermore, we can deduce
\begin{align*}
\bu
&\approx \bu+x_1x_2\cdots x_{j-1}\bp+x_1x_2\cdots x_{j-1}x_j\bu_{j_2}\\
&\approx \bu+{(x_1x_2\cdots x_{j-1}\bp)}^2+x_1x_2\cdots x_{j-1}x_j\bu_{j_2}&&(\text{by}~(\ref{42704}))\\
&\approx \bu+{(x_1x_2\cdots x_{j-1}\bp)}^2+x_1x_2\cdots x_{j-1}x_j\bu_{j_2}+x_1x_2\cdots x_{j-1}x_j\bp.
&&(\text{by}~(\ref{42707}), (\ref{42702}))
\end{align*}
So we can obtain $\bu \approx \bu+x_1x_2\cdots x_{j-1}x_j\bp$.

\textbf{Subcase 1.2.} $\ell(\bu_j)=1$ if $\bu_j \in \bu $ and $h_{Y_j}(\bq)=h_{Y_j}(\bu_j)$.
Then $ \bu_{j}=x_j.$
We shall show that there exists $\bu_{j_3}\in \bu$ such that $c(\bu_{j_3})\subseteq Y_j$.
If this is not true,
then we consider the semiring homomorphism $\varphi: P_f(X^+) \to S_{(4, 427)}$ defined by
$\varphi(x)=4$ if $x\in Y_j$, $\varphi(x)=3$ if $x=x_j$, and $\varphi(x)=2$ otherwise.
It is easy to see that $\varphi(\bu)=3$ and $\varphi(\bq)=1$, a contradiction.
So $c(\bu_{j_3})\subseteq Y_j$ for some $\bu_{j_3}\in \bu$.
We deduce
\begin{align*}
\bu
&\approx \bu+\bu_j+\bu_{j_3}+x_1x_2\cdots x_{j-1}\bp  \\
&\approx \bu+x_j+\bu_{j_3}+x_1x_2\cdots x_{j-1}\bu_{j_3}\bp &&(\text{by}~(\ref{42702}), (\ref{id24082725}), (\ref{42703}))\\
&\approx \bu+x_j+\bu_{j_3}+x_1x_2\cdots x_{j-1}\bu_{j_3}\bp+x_1x_2\cdots x_{j-1}\bu_{j_3}x_j\bp &&(\text{by}~(\ref{42709}))\\
&\approx \bu+x_j+\bu_{j_3}+x_1x_2\cdots x_{j-1}\bu_{j_3}\bp+x_1x_2\cdots x_{j-1}x_j\bp.
&&(\text{by}~(\ref{42702}), (\ref{id24082725}), (\ref{42703}))
\end{align*}
This implies $\bu\approx \bu+x_1x_2\cdots x_{j-1}x_j\bp$.

Take $j=m$. We obtain the identity $\bu\approx \bu+x_1x_2\cdots x_m\bp$. So we deduce
\[
\bu\approx \bu+x_1x_2\cdots x_m\bp\approx \bu+i(\bq)\bp
\stackrel{(\ref{42702}), (\ref{id24082725}), (\ref{42703})}\approx \bu+i(\bq)\approx \bu+\bq.
\]
This derives $\bu\approx \bu+\bq$.

\textbf{Case 2.} $D_\bq(\bu)\cap L_{\geq 2}(\bu)$ is empty. Then $D_\bq(\bu)\subseteq L_{1}(\bu)$.
If $\bu_i=x_1$ for some $\bu_i \in D_\bq(\bu)$, then
\[
\bu\approx \bu+\bu_i+\bu_k \approx \bu+x_1+x_1s(\bu_k) \stackrel{(\ref{42706})}\approx \bu+x_1^2+x_1s(\bu_k).
\]
The remaining steps are similar to Case 1.
If $\bu_i\neq x_1$ for all $\bu_i \in D_\bq(\bu)$,
then there exists $2\leq i\leq n$ such that $x_i \in D_\bq(\bu)$, but $x_1, \ldots, x_{i-1} \notin D_\bq(\bu)$.
We shall prove that there exists $\bu_s \in \bu$ such that either $x_jx_i$ is a prefix of $\bu_s$ for some $1\leq j\leq i$
or $\bu_s=x_is(\bu_s)$ and $ \ell(\bu_s)\geq 2$.
Consider the semiring homomorphism $\varphi: P_f(X^+) \to S_{(4, 427)}$ defined by
$\varphi(x_i)=3$, $\varphi(x)=4$ if $x\in Y_j$, and $\varphi(x)=2$ otherwise.
It is easy to see that $\varphi(\bu)=3$ and $\varphi(\bq)=1$, a contradiction.
If there exists $\bu_s \in \bu$ such that $x_jx_i$ is a prefix of $\bu_s$ for some $1\leq j\leq i$,
then $\bu_s=x_jx_i\bu_s'$ for some word $\bu_s'$ and so
\[
\bu\approx \bu+x_i+\bu_s \approx \bu+x_i+x_jx_i\bu_s' \stackrel{(\ref{42708})}\approx \bu+x_i+x_jx_i{\bu_s}'+x_jx_i.
\]
The remaining steps are similar to Case 1.
If there exists $\bu_s \in \bu$ such that $\bu_s=x_is(\bu_s)$ and $ \ell(\bu_s)\geq 2$,
then we have
\[
\bu\approx \bu+x_i+\bu_s \approx \bu+x_i+x_is(\bu_s) \stackrel{(\ref{42706})}\approx \bu+{x_i}^2+x_is(\bu_s).
\]
The remaining steps are similar to Case 1.
\end{proof}

 It is easy to see that $S_{(4, 417)}$ and $S_{(4, 427)}$  have dual multiplications. By Proposition $\ref{pro42701}$ we immediately deduce
\begin{cor}
The ai-semiring $S_{(4, 417)}$ is finitely based.
\end{cor}

\begin{pro}\label{pro42801}
$\mathsf{V}(S_{(4, 428)})$ is the ai-semiring variety defined by the identities
\begin{align}
x^3 &\approx x^2; \label{42801}\\
xy &\approx x^2y; \label{42802}\\
x^2y^2 &\approx (xy)^2; \label{42803}\\
x^2y^2 &\approx x^2y^2x^2; \label{42804}\\
x^2 &\approx x^2+x; \label{42805}\\
xyx &\approx xyx+xy; \label{42806}\\
xy^2 &\approx xy^2+xy; \label{42807}\\
x+xy &\approx x^2+xy; \label{42808}\\
xy+z & \approx xy+z+xz; \label{42809}\\
x+yxz &\approx y+yxz+yx; \label{42810}\\
x_1+x_2+x_3x_2x_4 &\approx x_1+x_2+x_3x_2x_4+x_3x_2x_1x_4.  \label{42811}
\end{align}
\end{pro}

\begin{proof}
It is easy to check that $S_{(4, 428)}$ satisfies the identities (\ref{42801})--(\ref{42811}).
In the remainder it is enough to show that every ai-semiring identity of $S_{(4, 428)}$
is derivable from (\ref{42801})--(\ref{42811}) and the identities defining $\mathbf{AI}$.
Let $\bu \approx \bu+\bq$ be such a nontrivial identity,
where $\bu=\bu_1+\bu_2+\cdots+\bu_n$ and $\bu_i, \bq\in X^+$, $1 \leq i \leq n$.
It is easy to see that $D_2$ is isomorphic to $\{2, 4\}$ and so $D_2$ satisfies  $\bu \approx \bu+\bq$.
This implies that there exists $\bu_{i_1} \in \bu$ such that $c(\bu_{i_1})\subseteq c(\bq)$.
Since $S_{58}$ is isomorphic to $\{1, 2, 3\}$,
it follows that $S_{58}$ satisfies $\bu \approx \bu+\bq$.
If $\ell(\bq)\geq 2$, then by Lemma \ref{lem5801} $L_{\geq 2}(\bu)\cap H_{\bq}(\bu)$ is nonempty
and so $\ell(\bu_k)\geq 2$ and $h(\bu_k)=h(\bq)$ for some $\bu_k \in \bu$.

If $|c(\bq)|=1$, then we can use the same approach in the proof of Proposition \ref{pro42701}
to deduce $\bu \approx \bu+\bq$.
Now let $|c(\bq)|\geq 2$.
By (\ref{42806}) and (\ref{42807}) we can deduce $\bq \approx \bq+i(\bq)$
and so $\bu \approx \bu+i(\bq)$ is satisfied by $S_{(4, 428)}$.
It is easy to see that $S_{41}$ is isomorphic to $\{1, 2, 4\}$
and so $S_{41}$ satisfies $\bu \approx \bu+\bq$.
Let $i(\bq)=x_1x_2\cdots x_m$, $m\geq 2$.
By Lemma \ref{lem4101} it follows that
for any $2\leq j \leq m$ and $Y_j=\{x_1, x_2, \ldots, x_{j-1}\}$,
there exists $\bu_{i_j} \in \bu$ such that $h_{Y_j}(\bu_{i_j})=h_{Y_j}(\bq)=x_j$.

The following technique (C) will be repeatedly used in the sequel.
Let $\bp$ be a word such that $|c(\bp)|\geq 2$, and let $\bp_1$ and $\bp_2$ be nonempty words
such that $c(\bp_2)\subseteq c(\bp)$.
We shall show that
\[
i(\bp)\bp_1\approx i(\bp)\bp_2\bp_1, ~\bp_2+i(\bp)\bp_1\approx \bp_2+i(\bp)\bp_1+\bp
\]
can be derived by (\ref{42801})--(\ref{42811}).
Indeed, let $i(\bp)=y_1y_2\cdots y_r$, $r\geq 2$. Then
\begin{align*}
\bp_2+i(\bp)\bp_1
&\approx \bp_2+y_1y_2\cdots y_r\bp_1\\
&\approx  \bp_2+y_1^2y_2^2\cdots y_r^2\bp_1 &&(\text{by}~(\ref{42802}))\\
&\approx \bp_2+{\bp}^2\bp^2_2\bp_1  &&(\text{by}~(\ref{42801}), (\ref{42804}), (\ref{42803}))\\
&\approx \bp_2+{\bp}\bp_2\bp_1&&(\text{by}~(\ref{42802}))\\
&\approx \bp_2+{\bp}\bp_2\bp_1+{\bp}\bp_2 &&(\text{by}~(\ref{42810}))\\
&\approx \bp_2+{\bp}\bp_2\bp_1+{\bp}\bp_2+{\bp}. &&(\text{by}~(\ref{42806}), (\ref{42807}))
\end{align*}
This implies $\bp_2+i(\bp)\bp_1\approx \bp_2+i(\bp)\bp_1+\bp$.

\textbf{Case 1.}
For any $2\leq j \leq m$, there exists $\bu_{i_j} \in \bu$ such that
$\bu_{i_j}=\bu_{i_j}'x_j\bu_{i_j}'' $
for some $\bu_{i_j}', \bu_{i_j}'' \in X^*$,
where $c(\bu_{i_j}') \subseteq Y_j$ and $\bu_{i_j}''$ is nonempty.
Then
\begin{align*}
\bu
&\approx \bu+\bu_k+\bu_{i_2}\\
&\approx \bu+x_1s(\bu_k)+\bu_{i_2}'x_2\bu_{i_2}''\\
&\approx\bu+x_1s(\bu_k)+\bu_{i_2}'x_2\bu_{i_2}''+x_1\bu_{i_2}'x_2\bu_{i_2}''&&(\text{by}~(\ref{42809}))\\
&\approx\bu+x_1s(\bu_k)+\bu_{i_2}'x_2\bu_{i_2}''+x_1x_2\bu_{i_2}''. &&(\text{by}~(\ref{42802}))
\end{align*}
This implies the identity
\[
\bu\approx \bu+x_1x_2\bu_{i_2}''.
\]
Furthermore, we have
\begin{align*}
\bu
&\approx \bu+x_1x_2\bu_{i_2}''+\bu_{i_3}\\
&\approx \bu+x_1x_2\bu_{i_2}''+\bu_{i_3}'x_3\bu_{i_3}''\\
&\approx \bu+x_1x_2\bu_{i_2}''+\bu_{i_3}'x_3\bu_{i_3}''+x_1x_2\bu_{i_3}'x_3\bu_{i_3}''&&(\text{by}~(\ref{42809}))\\
&\approx \bu+x_1x_2\bu_{i_2}''+\bu_{i_3}'x_3\bu_{i_3}''+x_1x_2\bu_{i_3}'x_3\bu_{i_3}''+x_1x_2x_3\bu_{i_3}''.
&&(\text{by}~(\ref{42806}), (\ref{42807}))
\end{align*}
This proves the identity
\[
\bu\approx \bu+x_1x_2x_3\bu_{i_3}''.
\]
Repeat this process and one can obtain the identity
\[
\bu\approx \bu+x_1x_2\cdots x_m\bu_{i_m}''.
\]
Take $\bp=\bq$, $\bp_1=\bu_{i_m}''$ and $\bp_2=\bu_{i_1}$.
By the technique (C)
we have that $\bu\approx \bu+{\bq}$ can be derived by (\ref{42801})--(\ref{42811}).

\textbf{Case 2.}
There exists $2\leq j \leq m$ such that $\bu_{i_j}=\bu_{i_j}'x_j$ for some $\bu_{i_j}' \in X^*$,
where $c(\bu_{i_j}')\subseteq Y_j$,
if $\bu_{i_j} \in \bu$ and $h_{Y_j}(\bu_{i_j})=h_{Y_j}(\bq)$. Choose a minimum $j$.

\textbf{Subcase 2.1.} $j<m$.
We shall show that
$c(\bu_{\ell_j}) \subseteq Y_j$ for some $\bu_{\ell_j} \in \bu$.
Indeed, suppose that this is not true.
Consider the semiring homomorphism $\varphi: P_f(X^+) \to S_{(4, 428)}$ defined by
$\varphi(x_j)=3$, $\varphi(x)=4$ if $x\in Y_{j}$, and $\varphi(x)=2$ otherwise.
It is easy to see that $\varphi(\bu)=3$ and $\varphi(i(\bq))=1$, a contradiction.
So $c(\bu_{\ell_j}) \subseteq Y_j$ for some $\bu_{\ell_j} \in \bu$.

Repeat the process in Case 1 and one can derive
\[
\bu\approx \bu+x_1x_2\cdots x_{j-1}\bu_{i_{j-1}}''.
\]
Choose $\bp=x_1x_2\cdots x_{j-1}$, $\bp_1=\bu_{i_{j-1}}''$ and $\bp_2=\bu_{\ell_j}$.
By the technique (C) we obtain that
\[
x_1x_2\cdots x_{j-1}\bu_{i_{j-1}}''\approx x_1x_2\cdots x_{j-1}\bu_{\ell_j}\bu_{i_{j-1}}''
\]
can be derived by (\ref{42801})--(\ref{42811}).
We now have
\begin{align*}
\bu
&\approx \bu+\bu_{i_j}+\bu_{\ell_j}+x_1x_2\cdots x_{j-1}\bu_{\ell_j}\bu_{i_{j-1}}''\\
&\approx \bu+\bu_{i_j}'x_j+\bu_{\ell_j}+x_1x_2\cdots x_{j-1}\bu_{\ell_j}\bu_{i_{j-1}}''\\
&\approx \bu+\bu_{i_j}'x_j+\bu_{\ell_j}+x_1x_2\cdots x_{j-1}\bu_{\ell_j}\bu_{i_{j-1}}''+
x_1x_2\cdots x_{j-1}\bu_{\ell_j}\bu_{i_j}'x_j\bu_{i_{j-1}}''. &&(\text{by}~(\ref{42811}))
\end{align*}
This proves the identity
\[
\bu\approx \bu+x_1x_2\cdots x_{j-1}\bu_{\ell_j}\bu_{i_j}'x_j\bu_{i_{j-1}}''.
\]
By the identities (\ref{42806}) and (\ref{42807}) we deduce
\[
x_1x_2\cdots x_{j-1}\bu_{\ell_j}\bu_{i_j}'x_j\bu_{i_{j-1}}''\approx x_1x_2\cdots x_{j-1}\bu_{\ell_j}\bu_{i_j}'x_j\bu_{i_{j-1}}''+
x_1x_2\cdots x_{j-1}x_j\bu_{i_{j-1}}''
\]
and so $\bu\approx \bu+x_1x_2\cdots x_{j-1}x_j\bu_{i_{j-1}}''$ is obtained.
We repeat this process and finally derive the identity
$\bu \approx  \bu+x_1x_2\cdots x_m\bu_{i_m}''$.

If $\bu_{i_m}''$ is nonempty, then the remaining steps are similar to Case 1.
If $\bu_{i_m}''$ is empty, then $\bu_{i_m}=\bu_{i_m}'x_m$ and $\bu \approx \bu+i(\bq)$ is obtained.
Suppose that $t(\bq)=x_m$ and that $m(t(\bq), \bq)=1$.
It is easy to see that the identities (\ref{42801})--(\ref{42804}) imply $i(\bq) \approx \bq$
and so $\bu \approx  \bu+\bq$ is derived.
If $t(\bq)\not=x_m$ or $m(t(\bq), \bq)\geq 2$,
then one can use the above substitution $\varphi$ to show that $c(\bu_{\ell_m})\subseteq Y_m$ for some $\bu_{\ell_m} \in \bu$.
Take $\bp=x_1x_2\cdots x_{m-1}$, $\bp_1=x_m$ and $\bp_2=\bu_{\ell_m}$.
By the technique (C) we have that
\[
x_1x_2x_3\cdots x_{m-1}x_m\approx x_1x_2\cdots x_{m-1}\bu_{\ell_m}x_m
\]
can be derived by (\ref{42801})--(\ref{42811}).
Furthermore, we deduce
\begin{align*}
\bu
&\approx \bu+\bu_{i_m}+\bu_{\ell_m}+x_1x_2\cdots x_{m-1}\bu_{\ell_m}x_m\\
&\approx \bu+\bu_{i_m}'x_m+\bu_{\ell_m}+x_1x_2\cdots x_{m-1}\bu_{\ell_m}x_m\\
&\approx \bu+\bu_{i_m}'x_m+\bu_{\ell_m}+x_1x_2\cdots x_{m-1}\bu_{\ell_m}x_m+x_1x_2\cdots x_{m-1}\bu_{\ell_m}\bu_{i_m}'x_mx_m. &&(\text{by}~(\ref{42811}))
\end{align*}
This derives the identity
\[
\bu \approx \bu+x_1x_2\cdots x_{m-1}\bu_{\ell_m}\bu_{i_m}'x_mx_m.
\]
By the identities (\ref{42806}) and (\ref{42807}) we deduce
\[
\bu \approx \bu+x_1x_2\cdots x_{m-1}x_mx_m.
\]
Choose $\bp=\bq$, $\bp_1=x_m$ and $\bp_2=\bu_{i_1}$.
By the technique (C) we have that
\[
\bu_{i_1}+x_1x_2\cdots x_{m-1}x_mx_m\approx \bu_{i_1}+x_1x_2\cdots x_{m-1}x_mx_m+\bq
\]
can be derived by (\ref{42801})--(\ref{42811}).

\textbf{Subcase 2.2.} $j=m$.
Then $\bu_{i_j}''$ is nonempty for all $2\leq j \leq m-1$ and $\bu_{i_m}''$ is empty.
Repeat the process in Case 1 and we obtain
\[
\bu \approx \bu+x_1x_2\cdots x_m\approx \bu+i(\bq).
\]
The remaining steps are similar to Subcase 2.1.
\end{proof}

\begin{cor}
The ai-semiring $S_{(4, 418)}$ is finitely based.
\end{cor}
\begin{proof}
It is easy to see that $S_{(4, 418)}$ and $S_{(4, 428)}$  have dual multiplications.
By Proposition $\ref{pro42801}$ we immediately deduce that $S_{(4, 418)}$ is finitely based.
\end{proof}

\section{Conclusion}
Combining the previous paper \cite{rlzc} we have answered the finite basis problem for $151$ ai-semirings of order four.
One can find that it is particularly important in this process to obtain
enough information about the identities of some ai-semirings of order three.
The remaining $715$ ai-semirings of order four will be divided into three classes according to their additive reducts.
The finite basis problem for most of them can be answered.
At present, it is very necessary to introduce some new techniques to complete this programme.

\qquad

\noindent
\textbf{Acknowledgements}
The authors thank their team members Junyang Liu, Zexi Liu, Qizheng Sun and Mengyu Yuan for discussions contributed to this paper.

\noindent
\textbf{Data availability}
Data sharing not applicable to this article as datasets were neither generated nor analysed.

\noindent
\textbf{Declarations}

\noindent
\textbf{Ethical Approval} It is not applicable to this article.

\noindent
\textbf{Competing interests} It is not applicable to this article.

\noindent
\textbf{Authors' contributions} It is not applicable to this article.

\noindent
\textbf{Funding} It is not applicable to this article.

\noindent
\textbf{Availability of data and materials} It is not applicable to this article.

\bibliographystyle{amsplain}

\begin{thebibliography}{99}
\bibitem{gpz} S. Ghosh, F. Pastijn, X.Z. Zhao,
Varieties generated by ordered bands I, \emph{Order} \textbf{22} (2005), no.\,2, 109--128.

\bibitem{jrz} M. Jackson, M.M. Ren, X.Z. Zhao, Nonfinitely based
ai-semirings with finitely based semigroup reducts, \emph{J. Algebra} \textbf{611} (2022), 211--245.

\bibitem{pas05} F. Pastijn, Varieties generated by ordered bands II, \emph{Order} \textbf{22} (2005), no.\,2, 129--143.

\bibitem{rlzc} M.M. Ren, J.Y. Liu, L.L. Zeng, M.L. Chen,
The finite basis problem for additively idempotent semirings of order four, I,
https://doi.org/10.48550/arXiv.2407.15342

\bibitem{rzw} M.M. Ren, X.Z. Zhao, A.F. Wang, On the varieties of ai-semirings satisfying $x^3 \approx x$,
\emph{Algebra Universalis} {\bf 77} (2017), no.\,4, 395--408.

\bibitem{sr} Y. Shao, M.M. Ren, On the varieties generated by ai-semirings of order two,
\emph{Semigroup Forum} {\bf 91} (2015), no.\,1, 171--184.

\bibitem{wrz}Y.N. Wu, M.M. Ren, X.Z. Zhao,
The additively idempotent semiring  $S^0_7$  is nonfinitely based,
\emph{Semigroup Forum} {\bf 108} (2024), no.\,2, 479--487.

\bibitem{zrc}X.Z. Zhao, M.M. Ren, S. Crvenkovi{\' c}, Y. Shao, P. \DH api{\' c},
The variety generated by an ai-semiring of order three,
\emph{Ural Math. J.} {\bf 6} (2020), no.\,2, 117--132.
\end{thebibliography}

\end{document}